\documentclass[a4paper, 11pt]{amsart}

\usepackage{latexsym}
\usepackage{amsthm, amssymb, amsmath, stmaryrd}
\usepackage{amscd} 
\usepackage{mathrsfs,array}
\usepackage{lmodern}

\usepackage{tabularx} 
\usepackage[all,cmtip,line]{xy}
\usepackage{enumerate}
\usepackage{hyperref}

\newtheorem{thm}{\textbf{Theorem}}[section]
\newtheorem{defn}[thm]{\textbf{Definition}}
\newtheorem{prop}[thm]{\textbf{Proposition}}
\newtheorem{lem}[thm]{\textbf{Lemma}}
\newtheorem{corollary}[thm]{\text{Corollary}}
\newtheorem{rem}[thm]{\textbf{Remark}}

\newtheorem*{thmn}{\textbf{Theorem}}

\def\Q{\mathbb{Q}}
\def\Z{\mathbb{Z}}
\def\C{\mathbb{C}}
\def\A{\mathbb{A}}
\def\R{\mathbb{R}}
\def\Gal{\operatorname{Gal}}

\def\GL{\operatorname{GL}}
\def\SL{\operatorname{SL}}

\def\proet{\mathrm{pro\acute{e}t}}

\def\cM{\mathcal{M}}
\newcommand{\mfM}{\mathfrak{M}}
\newcommand{\cF}{\mathcal{F}}
\newcommand{\cS}{\mathcal{S}}
\def\cO{\mathcal{O}}
\def\cX{\mathcal{X}}
\def\et{\mathrm{\acute{e}t}}
\def\cont{\mathrm{cont}}
\def\LT{\mathrm{LT}}
\def\HT{\mathrm{HT}}
\def\GH{\mathrm{GH}}
\def\Qbar{\overline{\mathbb{Q}}}
\def\xbar{\overline{x}}

\newcommand{\Fl}{{\mathscr{F}\!\ell}}
 
\newcommand{\Hom}{\operatorname{Hom}}
 
\newcommand{\Spf}{\operatorname{Spf}}
\newcommand{\spa}{\operatorname{Spa}} 
\newcommand{\Spa}{\operatorname{Spa}} 
\newcommand{\Map}{\operatorname{Map}}
\newcommand{\Ind}{\operatorname{Ind}}

\title{A quotient of the Lubin--Tate tower}
\author{Judith Ludwig}
\email{ludwig@mpim-bonn.mpg.de}

\begin{document}
\maketitle

\begin{abstract} In this article we show that the quotient $\cM_\infty/B(\Q_p)$ of the Lubin--Tate space at infinite level $\cM_\infty$ by the Borel subgroup of upper triangular matrices $B(\Q_p) \subset \GL_2(\Q_p)$ exists as a perfectoid space. As an application we show that Scholze's functor $H^i_\et(\mathbb{P}^1_{\C_p}, \cF_\pi)$ is concentrated in degree one whenever $\pi$ is an irreducible principal series representation or a twist of the Steinberg representation of $\GL_2(\Q_p)$.
\end{abstract}

\section{Introduction}
In \cite{scholzeLT}, Scholze constructs a candidate for the mod $p$ local Langlands correspondence for the group $\GL_n(F)$, where $n\geq 1$ and $F/\Q_p$ is a finite extension. His construction is purely local, satisfies some local-global compatibility and also gives a candidate for the mod $p$ Jacquet--Langlands transfer from $\GL_n(F)$ to $D^*$, the group of units in the central division algebra over $F$ with invariant $1/n$. 

Slightly more precisely, for any admissible smooth representation $\pi$ of $\GL_n(F)$ on an $\mathbb{F}_p$-vector space, Scholze constructs an \'etale sheaf $\cF_\pi$ on $\mathbb{P}_{\C_p}^{n-1}$ using the infinite level Lubin--Tate space $\cM_\infty$ and the Gross--Hopkins period morphism $\cM_\infty \rightarrow \mathbb{P}_{\C_p}^{n-1}$. The cohomology groups 
\[\cS^i(\pi):= H_\et^i(\mathbb{P}_{\C_p}^{n-1},\cF_\pi), i \geq 0,\]
are admissible $D^*$-representations which carry an action of $\Gal(\overline{F}/F)$ and vanish in degree $i > 2(n-1)$ (\cite[Theorem 1.1]{scholzeLT}).

Let $n=2$ and $F=\Q_p$, and denote by $B(\Q_p)\subset \GL_2(\Q_p)$ the Borel subgroup of upper triangular matrices. For two smooth characters $\chi_i: \Q_p^*\rightarrow \mathbb{F}_q^*, i=1,2$, let $\Ind^{\GL_2(\Q_p)}_{B(\Q_p)}(\chi_1,\chi_2)$ denote the parabolic induction. 

In this article we show the following vanishing result.
\begin{thmn}[\ref{mainthm}]
Let $\chi_i:\Q_p^* \rightarrow \mathbb{F}^*_q, i=1,2$, be two smooth characters. Then 
\[\cS^2\left(\Ind^{\GL_2(\Q_p)}_{B(\Q_p)}(\chi_1,\chi_2)\right) =0.\]
\end{thmn}

The main ingredient to prove the theorem is the construction of a quotient $\cM_\infty/B(\Q_p)$ of the infinite-level Lubin--Tate tower $\cM_\infty$ by the Borel subgroup $B(\Q_p)$ as a perfectoid space. 

In order to explain our results in more detail we introduce some notation. We fix a connected $p$-divisible group $H/\overline{\mathbb{F}}_p$ of dimension one and height two. Then for any compact open subgroup $K \subset \GL_2(\Q_p)$ we have the Lubin--Tate space $\cM_K$ of level $K$, which is a rigid analytic variety over $\C_p$. Varying $K$ one gets a tower $(\cM_{K})_{K\subset \GL_2(\Q_p)}$ with finite \'etale transition maps and in the limit one has a perfectoid space
\[\cM_\infty \sim \varprojlim_K \cM_K.\]
In particular, $\cM_\infty$ is an adic space and an object of Huber's ambient category $\mathcal{V}$. The group $\GL_2(\Q_p)$ acts on $\cM_\infty$. We have a decomposition 
\[\cM_\infty \cong \bigsqcup_{i \in \Z}\cM^{(i)}_\infty\]
and for any $i \in \Z$, there is a non-canonical isomorphism $\cM^{(i)}_\infty \cong \cM^{(0)}_\infty$. The stabilizer in $\GL_2(\Q_p)$ of $\cM^{(0)}_\infty$ is given by 
\[G':=\{g \in \GL_2(\Q_p) : \det(g) \in \Z_p^*\}.\]
As we recall in Section \ref{subsec: quot} we can form a quotient $\cM_\infty/B(\Q_p)$ in the category $\mathcal{V}$. The key to proving Theorem \ref{mainthm} is the following.

\begin{thmn}[\ref{secmainthm}] The object 
\[\cM_{\infty}/B(\Q_p)\]
of $\mathcal{V}$ is a perfectoid space. 
\end{thmn}

The proof of this theorem is done in three steps, the first step is the most difficult one.
There one constructs a candidate perfectoid space $\cM_B$ for the quotient $\cM^{(0)}_\infty/B(\Z_p)$ using infinite-level modular curves, the geometry of the Hodge-Tate period map and various other results of \cite{torsion}. 
\begin{thmn}[\ref{cover}]
 There exists a unique (up to unique isomorphism) perfectoid space 
\[\cM_B \sim \varprojlim \cM^{(0)}_{m}/B(\Z/p^m\Z).\]
\end{thmn}
In a second step one shows that the candidate space $\cM_B$ agrees with the object $\cM^{(0)}_\infty/B(\Z_p)$ in $\mathcal{V}$. 
Finally one gets from $\cM^{(0)}_\infty/B(\Z_p)$ to a perfectoid space $\cM_\infty/B(\Q_p)$ using the geometry of the Gross--Hopkins period morphism.

The Gross--Hopkins period morphism $\pi_{\GH}:\cM_\infty \rightarrow \mathbb{P}^1_{\C_p}$ factors through the quotient $\cM_\infty/B(\Q_p)$ and we get an induced map
\[\overline{\pi}_{\GH}: \cM_\infty/B(\Q_p)\rightarrow \mathbb{P}^1_{\C_p}\]
with many nice properties (cf.\ Section \ref{subsec:3.6}), e.g.\ it is quasi-compact. 
From a character $\chi:B(\Q_p)\rightarrow \mathbb{F}_q^*$ one can construct an \'etale local system $\cF_\chi$ on $\cM_\infty/B(\Q_p)$ such that  
\[\cF_{\pi}\cong \overline{\pi}_{\GH,*} \cF_\chi,\]
where $\pi:=\Ind^{\GL_2(\Q_p)}_{B(\Q_p)}(\chi)$. We show that all higher direct images $R^i\overline{\pi}_{\GH,*} \cF_\chi$, $i >0$, vanish (Proposition \ref{directimage}) and so
\[\cS^i\left(\Ind^{\GL_2(\Q_p)}_{B(\Q_p)}(\chi)\right) = H^i_{\et}(\cM_\infty/B(\Q_p), \cF_\chi).\]
We see that at the expense of a more complicated space, the sheaf has simplified a great deal. And although the space is more complicated, it is perfectoid. This fact and the good properties of $\overline{\pi}_{\GH}$ make it possible to show vanishing of $\cS^2\left(\Ind^{\GL_2(\Q_p)}_{B(\Q_p)}(\chi)\right)$
in Theorem \ref{mainthm}.

The article is structured as follows. In Section 2 we give some background on inverse limits of adic spaces, quotients of adic spaces and we introduce some auxiliary spaces that we need in the construction of $\cM_\infty/B(\Q_p)$. The quotient is constructed in Section 3. In the last section we prove the vanishing result in cohomological degree two. 

\textbf{Acknowledgments.} The author would like to thank Peter Scholze for suggesting the project underlying this article to her and for sharing his ideas which made this paper possible. She would like to thank him for this constant support, for his patience in answering her many questions and for the careful reading of the manuscript. The author would furthermore like to thank Ana Caraiani for many helpful conversations and Sug Woo Shin for pointing out an error at an earlier stage of the project. Finally the author is very grateful to the anonymous referee for many helpful comments and suggestions. The author was supported by the SFB/TR 45 of the DFG and her current affiliation, the Max Planck Institute for Mathematics. 

\subsection{Notation}
All adic spaces occurring are defined over $(K,\cO_K)$, where $K$ is a non-archimedean field, i.e., a non-discrete complete topological field $K$ whose topology is induced by a non-archimedean absolute value $|\cdot|:K\rightarrow \R_{\geq 0}$. Throughout we fix a topologically nilpotent unit $\varpi \in K^*$. In particular all adic spaces considered in this article are analytic. We will often use that if $X$ and $Y$ are affinoid analytic adic spaces and $f:X\rightarrow Y$ is a morphism, then the preimage of a rational subset $U\subset Y$ is a rational subset in $X$. As is usual, for an analytic adic space $X$ we identify its points $x \in X$ with their underlying maps $\spa(K,K^+)\rightarrow X$ without introducing separate notation. In the context of adic spaces we use the terminology introduced in \cite{berkeley}, i.e.\ Huber rings, Huber pairs, etc. 

We use the pro-\'etale site as defined in \cite{scholzepHT}, i.e., in a tower we want the transition maps to eventually be finite \'etale surjective.

\section{Background on adic spaces and quotients}

\subsection{Inverse limits of adic spaces}
\begin{defn}[cf.\ Def. 2.4.1 in \cite{SW}] Let $(X_m)_{m\in I}$ be a cofiltered inverse system of adic spaces with finite \'etale transition maps. Let $X$ be an adic space and let $p_m:X\rightarrow X_m$ be a compatible family of maps. Write 
\[X \sim \varprojlim X_m\]
if the map of underlying topological spaces $|X| \rightarrow \varprojlim |X_m| $ is a homeomorphism and if there is an open cover of $X$ by affinoid subsets $\spa(A,A^+) \subset X$ such that the map 
\[ \varinjlim_{\spa(A_m,A_m^+)\subset X_m} A_m \rightarrow A
\]
has dense image, where the direct limit runs over all open affinoids $\spa(A_m,A_m^+)\subset X_m$ over which $\spa(A,A^+)\subset X \rightarrow X_m$ factors.
\end{defn}

Recall the following result. 
\begin{prop}\label{unique} Let $(K,\mathcal{O}_K)$ be a perfectoid field, $(X_m)_{m \in I}$ a cofiltered inverse system of adic spaces with finite \'etale transition maps. Assume there is a perfectoid space $X$ over $\spa(K,\cO_K)$ such that $X \sim \varprojlim X_m$. Then for any affinoid uniform $(K,\cO_K)$-algebra $(R,R^+)$, we have 
\[X(R,R^+) = \varprojlim X_m(R,R^+).\]   
In particular, if $Y$ is a perfectoid space over $\spa(K,\cO_K)$ with a compatible system of maps $Y \rightarrow X_m$, then $Y$ factors over $X$ uniquely, making $X$ unique up to unique isomorphism.  
\end{prop}
\begin{proof} For $(R,R^+)$ affinoid perfectoid, this is \cite[Prop.\ 2.4.5]{SW}. However the same proof works for any affinoid $(K,\cO_K)$-algebra $(R,R^+)$ with $R$ uniform or what is equivalent, with $R^+$ $\varpi$-adically complete. 
\end{proof}

Consider also the following stronger notion.
\begin{defn} Let $(K,\cO_K)$ be a non-archimedean field. Let 
\[(X_m=\spa(R_m, R^+_m))_{m\in I}\]
be a cofiltered inverse system of affinoid adic spaces over $\spa(K,\cO_K)$ with finite \'etale transition maps. Assume $X=\spa(R,R^+)$ is an affinoid adic space over $\spa(K,\cO_K)$ with a compatible family of maps $p_m: X \rightarrow X_m$.
We write 
\[ X \approx \varprojlim X_m
\] 
if $R^+$ is the $\varpi$-adic completion of $\varinjlim_m R^+_m$. 
\end{defn}
Obviously if $X \approx \varprojlim X_m$ then also $X \sim \varprojlim X_m$. The reason for introducing this stronger notion is the following. When $X \sim \varprojlim X_m$, we do not know whether an arbitrary affinoid open $U=\spa(R,R^+) \subset X$ is the preimage $p_m^{-1}(U_m)$ of an open affinoid subset $U_m =\spa(R_m,R_m^+) \subset X_m$ for sufficiently large $m$ and if it satisfies $R^+ \cong (\varinjlim R^+_m)^{\wedge}$. 
For this reason some of the arguments below become slightly technical.  

We can however always pass to rational subsets. 
\begin{lem} \label{1.1} Let $(X_m=\spa(R_m,R^+_m))_{m\in I}$ be a cofiltered inverse system of affinoid analytic spaces over $\spa(K,\cO_K)$. Assume $X=\spa(R,R^+)$ is an affinoid perfectoid space over $\spa(K,\cO_K)$, with compatible maps $p_m:X \rightarrow X_m$ and such that $X \approx \varprojlim X_m$. Let $U\subset X$ be a rational subset. Then $U=p_{m_0}^{-1}(U_{m_0})$ for a rational subset $U_{m_0} \subset X_{m_0}$ and $m_0$ sufficiently large and 
\[U \approx \varprojlim U_{m_0} \times_{X_{m_0}} X_m.\] 
\end{lem}
\begin{proof} We write $U=X\left(\frac{f_1}{g},\cdots, \frac{f_n}{g}\right)$ for some elements $f_i,g \in R$ such that $f_1,..., f_n$ generate $R$. Then for some small open neighbourhood $V$ of $0\in R$, 
the rational subset defined by any choice of functions $f'_i \in f_i +V$, $g'\in g+V$ agrees with $U$. As $\varinjlim R_m$ is dense in $R$, we may assume the $f'_i$ and $g'$ are elements of $R_{m_0}$ for some $m_0$. Define 
\[U_{m_0}:= X_{m_0}\left(\frac{f'_1}{g'}, \cdots, \frac{f'_n}{g'} \right)\cong  \spa\left(R_{m_0}\left\langle \frac{f'_1}{g'}, \cdots, \frac{f'_n}{g'}\right\rangle, \widehat{B}_{m_0}\right),\]
where $\widehat{B}_{m_0}$ is the completion of the integral closure $B_{m_0}$ of $R_{m_0}^+[\frac{f'_1}{g'},\cdots,\frac{f'_n}{g'}]$ in $R_{m_0}[g'^{-1}]$. 
Then 
\[U_m:= U_{m_0}\times_{X_{m_0}} X_m \cong \spa\left(R_{m}\left\langle \frac{f'_1}{g'}, \cdots, \frac{f'_n}{g'}\right\rangle, \widehat{B}_{m}\right)\]
with $B_m$ the obvious ring and
\[U \cong X\left(\frac{f'_1}{g'}, \cdots, \frac{f'_n}{g'}\right) \cong \spa\left(R\left\langle \frac{f'_1}{g'}, \cdots, \frac{f'_n}{g'}\right\rangle, \widehat{B}\right).\]
As $(R,R^+)$ is perfectoid, $\widehat{B}$ is $\varpi$-adically complete. We see that
\[U \cong X\left(\frac{f'_1}{g'}, \cdots, \frac{f'_n}{g'}\right) \cong U_{m_0}\times_{X_{m_0}} X \approx \varprojlim U_{m_0}\times_{X_{m_0}} X_m\] 
as $\varinjlim B_m \subset \widehat{B}$ is dense in the $\varpi$-adic topology. 
\end{proof}

In this article, we call an analytic adic space $X$ over a non-archimedean field $(K,K^+)$ \emph{partially proper} if $X$ is quasi-separated and if for all affinoid $(K,K^+)$-algebras $(k,k^+)$ with $k$ a non-archimedean field and $k^+$ a valuation ring, 
\[X(k,k^+) = X(k,\cO_k). \]
As we want to use this concept in the context of perfectoid spaces, there is no finiteness condition involved. So our definition is slightly weaker than Definition 1.3.3 in \cite{huber}.

From Proposition \ref{unique} above one immediately concludes
\begin{prop}\label{univspec} Let $X \sim \varprojlim X_m$ and assume $X$ is perfectoid. If there exists $m_0$ such that $X_m$ is partially proper for all $m\geq m_0$, then $X$ is partially proper. 
\end{prop}

\subsection{Quotients in Huber's category $\mathcal{V}$}\label{subsec: quot}

The category $\mathcal{V}$ is defined in \cite{huber1} as follows. Its objects are triples 
\[X=(|X|, \cO_X, \{|\cdot|_x\}_{x \in |X|}),
\] 
where $|X|$ is a topological space, $\mathcal{O}_X$ is a sheaf of complete topological rings on $|X|$ and $|\cdot|_x$ is a valuation on the stalk $\cO_{X,x}:= \varinjlim_{x \in U} \cO_{X}(U)$. A morphism $X \rightarrow Y$ in $\mathcal{V}$ is a pair $(f, \varphi)$ consisting of a continuous map $f:|X| \rightarrow |Y|$ and a morphism $\varphi:\cO_Y \rightarrow f_*\cO_X$ of sheaves of topological rings, such that for every $x \in X$, the induced ring homomorphism $\varphi_x:\cO_{Y,f(x)}\rightarrow \cO_{X,x}$ is compatible with the valuations $|\cdot|_x$ and $|\cdot|_{f(x)}$.

Let $X \in \mathcal{V}$ be an object with a right action of a group $G$, i.e., with a homomorphism  
\[G^{op} \rightarrow Aut_{\mathcal{V}}(X).\]  

Consider the natural continuous and open map $p: |X| \rightarrow |X|/G$.
Define the triple 
\[
X/G:=(|X|/G, (p_*\cO_X)^G, \{|\cdot|_{\overline{x}}\}_{\overline{x} \in |X|/G})
\]
where the valuations $|\cdot|_{\overline{x}}$ are defined as 
\[ |\cdot|_{\overline{x}}:\cO_{|X|/G, \overline{x}} \rightarrow \cO_{X,x} \stackrel{|\cdot|_x}{\rightarrow} \Gamma_x\cup \{0\}
\]
where $|\cdot|_x$ is the valuation of any preimage $x \in p^{-1}(\overline{x})$. 
Then this defines an object of~$\mathcal{V}$. Indeed $(p_*\cO_X)^G$ is again a sheaf as taking $G$-invariants is left exact. Furthermore for any open subset $U \subset |X|/G$, 
\[(p_*\cO_X)^G(U)= \cO_X(p^{-1}(U))^G \subset \cO_X(p^{-1}(U))
\] 
is closed, therefore complete. 
One checks that $X/G$ is the categorical quotient of $X$ in $\mathcal{V}$. 

\subsection{Quotients of adic spaces}
\label{subsec:quotients}

We specialize to the case we are interested in, which is when $X$ is an analytic adic space over $\spa(K,\cO_K)$ for a non-archimedean field $K$, and mostly either a rigid space or a perfectoid space. In this section we also assume that $K$ is of characteristic zero.
Let $G$ be a locally profinite group. Below we want our actions to be continuous in the following sense. 

\begin{defn}[cf.\ {\cite[Def.\ 2.1]{scholzeLT}}]\label{1.8} Let $G$ be a locally profinite group and let $X$ be an analytic adic space. An action of $G$ on $X$ is said to be continuous if there exists a cover of $X$ by affinoid open subspaces $\spa(A,A^+)\subset X$ invariant under the action of open subgroups $H_A\subset G$ such that the action morphism $H_A\times A \rightarrow A$ is continuous.
\end{defn}

\begin{lem} Let $X$ be an analytic adic space and $G$ a finite group with an action of $G$ on $X$. Assume that $X$ has a cover by $G$-stable affinoid open subsets $\spa(A,A^+)$. Then $X/G$ is an analytic adic space. The quotient $\spa(A,A^+)/G$ is affinoid and given by $\spa(A^G, {A^+}^G)$.  
\end{lem}
\begin{proof} This is Theorem 1.1 and Theorem 1.2 in \cite{hansen2}.
\end{proof}

\begin{lem} Let $(K,\cO_K)$ be an affinoid perfectoid field. Let $G$ be a finite group acting on an affinoid perfectoid space $X=\spa(R,R^+)$ for some affinoid perfectoid $(K,\cO_K)$-algebra $(R,R^+)$. Then $X/G$ is again affinoid perfectoid and given by $\spa(R^G, {R^+}^G)$.
\end{lem}
\begin{proof} This is Theorem 1.4 in \cite{hansen2}.
\end{proof}

\begin{lem}\label{free}
Let $X$ be a rigid space or a perfectoid space over $\spa(K,\cO_K)$ and assume a finite group $G$ acts on $X$. Assume furthermore that $X\rightarrow X/G$ is finite \'etale Galois with Galois group $G$. Then
\[X\times G \cong X\times_{X/G} X.\]
\end{lem}
\begin{proof}We may assume $X=\spa(A,A^+)$ is affinoid. As $X \rightarrow X/G$ is a finite \'etale Galois cover with Galois group $G$ the map 
\begin{eqnarray*}
\varphi: A \otimes_{A^G} A &\rightarrow &\prod_{g \in G} A_{(g)}\\
a \otimes b &\mapsto & (a (g\cdot b))
\end{eqnarray*}
is an isomorphism. 
The image of $A^+\otimes_{(A^+)^G} A^+$ is contained in $\prod A^+$ and as $A^+$ is integrally closed in $A^\circ$, so is the image of the integral closure $(A^+\otimes_{(A^+)^G}A^+)^{int}$, in fact $\varphi((A^+\otimes_{(A^+)^G}A^+)^{int}) = \prod A^+$ as surjectivity can be checked componentwise.   
\end{proof}	

\subsection{Some auxiliary spaces}

\begin{prop}\label{sim} Let $(R,R^+)$ be a perfectoid $(K, \cO_K)$-algebra. Assume $P$ is a profinite set. Then
\[(S= \Map_{\mathrm{cont}}(P, R), \ S^+=\Map_{\mathrm{cont}}(P, R^+))\]
is again a perfectoid $(K,\cO_K)$-algebra. 
Let $X= \spa(R,R^+)$, then we define
\[X\times \underline{P} := \spa(S,S^+).\]
If we write $P = \varprojlim_m P_m$ for finite sets $P_m$, then 
\[X \times \underline{P} \approx \varprojlim X \times P_m. \]
\end{prop}

\begin{proof} One easily checks that the pair $(S,S^+)$ is perfectoid $(K,\cO_K)$-algebra.
Define $(S_m, S^+_m):= (\mathrm{Map}(P_m,R), \mathrm{Map}(P_m,R^+)) $. Then 
\[\varinjlim S_m = \mathrm{Map}_{loc.cst}(P,R),\]
and 
\[\varinjlim S^+_m = \mathrm{Map}_{loc.cst}(P,R^+).\]
We have that $\varinjlim S^+_m$ is dense in $S^+$ in the $\varpi$-adic topology. 
Now 
\[|\spa(S,S^+)| \cong \varprojlim |\spa(S_m, S^+_m)| \cong \varprojlim |X|\times P_m \cong |X|\times P.\] 
\end{proof}

\begin{prop}\label{rep} Keep the notation from above and let $Y$ be a perfectoid space, then 
\[\Hom(Y, X \times \underline{P}) = \Hom(Y,X) \times \mathrm{Map}_{\mathrm{cont}}(|Y|,P).\]
\end{prop}
\begin{proof} This follows from the previous Proposition and Proposition \ref{unique}. Namely we have
\begin{eqnarray*} \Hom(Y,X\times \underline{P}) &=& \varprojlim \Hom(Y,X\times P_m) \\
&=& \varprojlim (\Hom(Y,X) \times \mathrm{Map}(|Y|,P_m)) \\
&=& \Hom(Y, X) \times \varprojlim \mathrm{Map}(|Y|,P_m)\\
&=& \Hom(Y, X) \times \mathrm{Map}_{\mathrm{cont}}(|Y|,P).
\end{eqnarray*}  
\end{proof}

\begin{lem} We keep the notation from the Proposition \ref{sim}. Consider the rational subset $U= X\left(\frac{f_1}{g},\cdots, \frac{f_n}{g}\right)$ for some elements $f_i,g \in R$ such that $f_1,\dots, f_n$ generate $R$. Then 
\[(X\times \underline{P})\left(\frac{\mathbf{f}_1}{\mathbf{g}},\cdots, \frac{\mathbf{f}_n}{\mathbf{g}}\right) \cong U\times \underline{P},\]
where $\mathbf{f}_i \in \Map_{\mathrm{cont}}(P, R)$ (resp.\ $\mathbf{g} \in \Map_{\mathrm{cont}}(P, R)$) is the constant map with value $f_i \in R$ (resp.\ $g \in R$).  
\end{lem}
\begin{proof} It is clear that the underlying topological spaces agree. We show the corresponding perfectoid algebras are isomorphic by showing that the Huber pair
\[(S', S'^+):=\left(\Map_{\mathrm{cont}}\left(P, R\left(\frac{f_1}{g},\cdots, \frac{f_n}{g}\right)\right), \Map_{\mathrm{cont}}\left(P, R^+\left(\frac{f_1}{g},\cdots, \frac{f_n}{g}\right)\right) \right)\]
satisfies the same universal property as the Huber pair 
\[\left(S\left(\frac{\mathbf{f}_1}{\mathbf{g}},\cdots, \frac{\mathbf{f}_n}{\mathbf{g}}\right) , S^+\left(\frac{\mathbf{f}_1}{\mathbf{g}},\cdots, \frac{\mathbf{f}_n}{\mathbf{g}}\right)\right).\]
Clearly we have a map
\[\left(S\left(\frac{\mathbf{f}_1}{\mathbf{g}},\cdots, \frac{\mathbf{f}_n}{\mathbf{g}}\right) , S^+\left(\frac{\mathbf{f}_1}{\mathbf{g}},\cdots, \frac{\mathbf{f}_n}{\mathbf{g}}\right)\right) \rightarrow (S',S'^+).\]
Now let $(A,A^+)$ be perfectoid and assume we have a morphism of Huber pairs 
\[(S,S^+) \rightarrow (A,A^+)\]  
such that map $\varphi: \spa(A,A^+) \rightarrow \spa(S,S^+)$ factors over $(X\times \underline{P})\left(\frac{\mathbf{f}_1}{\mathbf{g}},\cdots, \frac{\mathbf{f}_n}{\mathbf{g}}\right)$. Let $p: X \times \underline{P} \rightarrow X$ denote the projection map (corresponding to the map $(R,R^+) \rightarrow (S,S^+)$ that sends an element $r \in R$ to the constant map with value $r$). Then the morphism 
\[p \circ \varphi : Y:=\spa(A,A^+) \rightarrow \spa(R,R^+)\]
factors over $U$. By the universal property of rational subsets we get a corresponding unique morphism of Huber pairs
\[\varphi_U:\left(R\left(\frac{f_1}{g},\cdots, \frac{f_n}{g}\right), R^+\left(\frac{f_1}{g},\cdots, \frac{f_n}{g}\right)\right) \rightarrow (A,A^+).\]
By the last proposition, $\varphi$ corresponds to a pair 
\[(\varphi_X, \varphi_P) \in \Hom(Y,X) \times \mathrm{Map}_{\mathrm{cont}}(|Y|,P).\]
Then $(S, S^+)\rightarrow (A,A^+)$ factors over the unique morphism 
\[ (S',S'^+) \rightarrow (A,A^+) \]
corresponding to $(\varphi_U,\varphi_P) \in \Hom(Y,U) \times \mathrm{Map}_{\mathrm{cont}}(|Y|,P)$.
\end{proof}

We see that we can glue the construction $X \times \underline{P}$ locally on $X$, so we can build a perfectoid space $X \times \underline{P}$ from any perfectoid space $X$.
Even more generally: Given a perfectoid space $X$ and a locally profinite set $Q$ there is a unique perfectoid space 
\[X \times \underline{Q}\]
such that for any perfectoid space $Y$, 
\[\Hom(Y, X \times \underline{Q}) = \Hom(Y,X) \times \mathrm{Map}_{\mathrm{cont}}(|Y|,Q).\]
If we write $Q = \bigsqcup_{i \in I} P_i$ for profinite sets $P_i$, then 
\[X \times \underline{Q} \cong \bigsqcup_i X\times \underline{P}_i.\]

\begin{lem}\label{cohpt} Let $x=\spa(K,K^+)$ with $K$ algebraically closed and $K^+\subseteq K$ an open and bounded valuation subring. Let $P= \varprojlim P_m$ be a profinite set. Then for any sheaf $\cF$ of abelian groups on $(x\times \underline{P})_{\et}$
\[H^i_{\et}(x\times \underline{P}, \cF) = 0\]
for all $i > 0 $. 
\end{lem}
\begin{proof} By Proposition \ref{sim} we have 
\[x\times \underline{P} \approx \varprojlim x \times P_m.\]
Any sheaf $\cF$ on $(x\times \underline{P})_{\et}$ can be written as a filtered colimit $\varinjlim_{j\in J} \mathcal{F}^j$ of sheaves $\mathcal{F}^j$ arising as the pullback of a system of sheaves $\cF^j_m$ on $(x\times P_m)_{\et}$ \cite[VI, 8.3.13]{SGA4}.
The topos $(x\times \underline{P})_{\et}$ is coherent, so for all $i\geq 0$
\[H^i_{\et}(x\times \underline{P}, \cF) \cong \varinjlim H^i_{\et}(x\times \underline{P}, \cF^j).\]
But by \cite[Theorem 2.4.7]{SW}
\[H^i_{\et}(x\times \underline{P}, \cF^j) \cong \varinjlim H^i_{\et}(x\times P_m, \cF_m^j)\]
and the latter terms are all zero for $i>0$ as $x\times P_m \cong \bigsqcup_{p\in P_m} x$ is a finite disjoint union of geometric points. 
\end{proof}

Now let $H\subset G$ be a profinite subgroup of a locally profinite group $G$ and let $X$ be a perfectoid space equipped with a continuous right action of $H$. 
Then $H$ acts on $X \times \underline{G}$ as follows. Firstly $H$ acts on the topological space $|X\times \underline{G}|$ via 
\[(x,g) \cdot h:= (xh,h^{-1}g)\]
for $(x,g) \in |X\times \underline{G}|$ and $h \in H$ and on the structure sheaf via 
\begin{eqnarray*}
h: \cO_{X\times \underline{G}}(U \times \underline{H'}) &\rightarrow& h_*\cO_{X\times \underline{G}}(U \times \underline{H'})=\cO_{X\times G}(U\cdot h^{-1} \times \underline{hH'})\\
  (h\cdot f)(g)&=& h \cdot f(h^{-1}g)
\end{eqnarray*} 
for any open $U\times \underline{H'} \subset |X\times \underline{G}|$ with $U=\spa(A,A^+)$ affinoid perfectoid and $H'$ an open coset. 

Consider the quotient 
\[X \times^H G := (X \times \underline{G})/H \in \mathcal{V}.\]
Note that $X \times^H G$ comes equipped with a right action of $G$ which is given by $ [(x,g')]\cdot g := [(x, g'g)]$ for $g \in G$ and $[(x,g')] \in (|X|\times G)/H$ on the topological spaces and in the natural way on the structure sheaf, i.e., by right translation of the function. 

\begin{prop}\label{nh} Using the notation above, there is an isomorphism in $\mathcal{V}$  
 \[(X\times^H G)/G \cong X/H. \]
\end{prop}
\begin{proof}
There is a natural continuous map of topological spaces $f:X \times^H G  \rightarrow X/H$ given by $[(x,g)] \mapsto [x]$. This is well defined as for $(y,g') \in [(x,g)]$ there exists $h \in H$ such that $y=xh, g'=h^{-1}g$. This map factors through the quotient $(X \times^H G)/G \rightarrow X/H$. One immediately checks it is bijective and open, so $f$ is a homeomorphism. Consider the commutative diagram
\[\xymatrix{
 &\ar[ld]^{q_1} X \times \underline{G} \ar[d]^{q}\ar[r]^p & X \ar[d]^{r}\\
X\times^H G \ar[r]^{q_2}& (X \times^H G)/G \ar[r]^f & X/H \\
}\]
Then for $U\subset |X|/H$ open, $\cO_{X/H}(U) = \cO_X(r^{-1}(U))^H$. On the other hand 
\begin{eqnarray*}
\cO_{(X\times^H G)/G}(f^{-1}(U)) &=& \cO_{(X\times^H G)}(q_2^{-1}(f^{-1}(U)))^G   \\
&=& (\cO_{X\times \underline{G}}((f\circ q)^{-1}(U))^{H})^{G} \\
&=& (\cO_{X\times \underline{G}}((f\circ q)^{-1}(U))^{G})^{H} \\
&=& (\cO_{X\times \underline{G}}(({r}^{-1}(U)\times \underline{G}))^{G})^H = \cO_X(r^{-1}(U))^H.
\end{eqnarray*}  
Compatibility of the valuations on the stalks is clear. 
\end{proof}

\begin{prop} Assume $H \subseteq G$ is open and profinite. Using the notation above, choose a set of representatives $g_i, i \in I$, of $H\backslash G$. The natural map 
\[ \xymatrix{ \bigsqcup_{g_i, i \in I} X_{(g_i)} \ar[rd]^f \ar[r] & X \times \underline{G} \ar[d]\\
&  X \times^H G 
}\]
is an isomorphism.
\end{prop}
\begin{proof} 
Define $X^{H\backslash G}:=\bigsqcup_{g_i, i \in I} |X|_{(g_i)}$. There is a continuous map on topological spaces 
\[f: X^{H \backslash G} \rightarrow |X\times^H G|,
\]
sending $x \in |X|_{(g_i)}$ to $(x,g_i)$. One easily checks that it is bijective and open, therefore a homeomorphism. 

For any $i \in I$, there is a morphism in $\mathcal{V}$, 
\[\varphi_i: X \rightarrow X\times \underline{G}, \ x \mapsto (x,g_i).
\]
On a basis it is given as follows:
If $U = V \times \underline{H_V}\cdot g_i \subset |X\times \underline{G}|$ open, with $V=\spa(A,A^+)$ affinoid perfectoid and $H_V$-invariant and such that $H_V\times A \rightarrow A $ is continuous, then the morphism of sheaves $\varphi_i^\sharp$ evaluates to  
\[\cO_{X\times \underline{G}}(V\times \underline{H_V} g_i) \rightarrow \cO_{X}(V), f \mapsto f(g_i).\]
Note that the map has a section given by sending $r \in \cO_X(V)$ to the function $ h g_i \mapsto h\cdot r$. 
This gives an isomorphism 
\[\cO_X(V) \rightarrow \cO_{X\times \underline{G}}(V\times \underline{H_V} g_i)^{H_V}.
\]
For $U= V \times \underline{H'} \subset |X\times \underline{G}|$ with $g_i \notin H'$, the corresponding map between sheaves is the zero map.  

We now show that the map $f$ induces an isomorphism of structure sheaves 
\[\cO_{X\times^H G} \cong f_* \cO_{X^{H\backslash G}}.\]
Denote by 
\[ p: |X| \times G \rightarrow |X \times^H G|
\] 
the open projection map.
Let $V \subset |X\times^H G| $ be an open subset. We may assume that $V= f(U)$ and $U \subset X_{(g_i)} $ for some $i$ and that $U= \spa(B,B^+)$ is affinoid open. Then there is $H_U \subset H$ normal and open such that $H_U$ leaves $U \subset X_{(g_i)}$ invariant and such that $H_U\times B \rightarrow B$ is continuous.
We have a decomposition 
\[p^{-1}(V) = \bigsqcup_{finite} U h_k \times H_U h_k^{-1}g_i,\] where $H/H_U = \bigcup H_U h_k$. Putting everything together we get
\begin{eqnarray*}
\cO_{X\times^H G}(V) &=& \cO_{X\times \underline{G}}(p^{-1}(V))^H \\
&=& (\prod\cO_{X\times \underline{G}}(U h_k \times H_U h_k^{-1}g_i))^H \\
&=& \cO_{X\times \underline{G}}(U\times H_U g_i)^{H_U}\\
&=& \cO_{X^{H\backslash G}}(U).
\end{eqnarray*}
\end{proof}
Let $Y$ be a perfectoid space with a continuous right action of a locally profinite group $G$. We define an action morphism
\[ \alpha: Y \times \underline{G} \rightarrow Y, (y,g) \mapsto yg
\]
as follows. Let $U= \spa(A,A^+)$ be an affinoid perfectoid open subspace of $Y$, and $H_U~\subset~G$ an open subgroup leaving $U$ invariant. Choose coset representatives $g_i, i~\in~I,$ for $H_U\backslash G$.  Then the natural map
\begin{eqnarray*} \bigsqcup_{i \in I} {(U \times \underline{H}_U)_{g_i}} &\rightarrow& U \times \underline{G} \\
(y,h)_{g_i} &\mapsto& (y, h g_i) 
\end{eqnarray*}
is an isomorphism. Then 
\[\alpha|_{(U \times \underline{H}_U)_{g_i}} : (U \times \underline{H}_U)_{g_i} \rightarrow Y\]
is defined as the composite $ g_i \circ \alpha_{H_U}$, where $\alpha_{H_U}: (U \times \underline{H}_U)_{g_i} \rightarrow U$ is the map induced by 
\[A \rightarrow  \Map_{\mathrm{cont}}(H_U, A), a \mapsto  (h \mapsto h \cdot a)
\]
and $g_i: U \rightarrow Y$ is the morphism from the action of $G$ on $Y$. 

\begin{lem}\label{immersion} Let $G$ be a locally profinite group, $H\subset G$ an open subgroup. Assume $f:X\rightarrow Y$ is a morphism between perfectoid spaces, that $H$ acts continuously on $X$ and $G$ on $Y$ and that $f$ is $H$-equivariant. Then there exists a unique $G$-equivariant map 
\[f': X\times^H G \rightarrow Y\]
such that $f$ is the composite of $f'$ and the natural inclusion $X = X \times^H H \hookrightarrow X \times^H G$. 
If $f$ is an open immersion and $|f'|: |X\times^H G| \rightarrow |Y|$ is an injection, then $f'$ is also an open immersion. 
\end{lem}
\begin{proof} Consider the morphism $(f,id):X\times \underline{G} \rightarrow Y\times \underline{G} $ and the action morphism $\alpha: Y \times \underline{G} \rightarrow Y$, $(y,g)\mapsto yg$ defined above. The composition is $H$-equivariant for the trivial action of $H$ on $Y$, therefore factors through the quotient $X\times^H G$ which gives the morphism $f'$.
To show the last claim, it suffices to show that the restriction $f'$ to each copy $X_{(g_i)}$ of $X$ is an open immersion.
But $g_i:Y \rightarrow Y$ is an isomorphism and $f':X_{(g_i)}\rightarrow Y$ agrees with $ g_i\circ f$. 
\end{proof}

\section{The quotient $\cM_{\infty}/B(\Q_p)$}

\subsection{The Lubin--Tate tower}

Let $k=\overline{\mathbb{F}}_p$ and denote by $W(k)$ the Witt vectors of $k$. Fix a connected $p$-divisible group $H/k$ of dimension one and height two. 
\begin{defn} Let $\operatorname{Nilp}_{W(k)}$ be the category of $W(k)$-algebras $R$ in which $p$ is nilpotent. 
A \emph{deformation} of $H$ to $R \in \operatorname{Nilp}_{W(k)}$ is a pair $(G,\rho)$ where $G$ is a $p$-divisible group over $R$ and 
\[\rho: H \otimes_{k} R/p \rightarrow G\otimes_{R}R/p\] 
is a quasi-isogeny. 
Define the functor
\begin{eqnarray*} \operatorname{Def_H}: \operatorname{Nilp}_{W(k)} &\rightarrow& Sets\\
R &\mapsto& \{(G,\rho): \text{ deformation of } H \}/_{\cong}.
\end{eqnarray*}
\end{defn}

\begin{thm}[\cite{RZ}] The functor $\operatorname{Def}_H$ is representable by a formal scheme $\mfM /W(k)$. We have a decomposition 
\[\mfM \cong \bigsqcup_{i \in \Z} \mfM^{(i)}\] 
according to the height $i$ of the quasi-isogeny and non-canonically 
\[\mfM^{(i)}\cong \Spf(W(k)[[t]]).\] 
\end{thm}
Let $\cM_{0}:=\mfM_\eta^{ad} \times_{\breve{\Q}_p} \C_p$ be the (base change to $\C_p$ of the) adic generic fibre of $\mfM$.
For an integer $m\geq 0$ let 
\[K(m):=\mathrm{ker}(\GL_2(\Z_p)\rightarrow \GL_2(\Z/p^m\Z)).\]
One can introduce level structures to get spaces $\cM_{m}:=\cM_{K(m)}$ and finite \'etale Galois covers $\cM_{m}\rightarrow \cM_{0}$ with Galois group $\GL_2(\Z/p^m\Z)$. More generally one can construct spaces $\cM_K$ for any compact open subgroup $K \subset \GL_2(\Q_p)$ (see \cite[Section 5.34]{RZ}).
For $i \in \Z$ and a compact open subgroup $K\subseteq \GL_2(\Z_p)$ let 
\[\cM_K^{(i)}:= \cM_K \times_{\cM_0} \cM_0^{(i)}\]
so that we have a decomposition 
\[\cM_K \cong \bigsqcup_{i\in \Z} \cM^{{(i)}}_K.\] 

For two compact open subgroups $K\subset K'$ of $\GL_2(\Q_p)$ such that $K$ is normal in~$K'$ there
is a finite \'etale Galois cover  
\[\cM_{K} \rightarrow \cM_{K'} \]
with Galois group $K'/K$. 
In particular the spaces $(\cM_K)_{K\subset \GL_2(\Q_p)}$ form an inverse system. In the limit we get a perfectoid space.
\begin{thm}[{\cite[6.3.4]{SW}}] There exists a unique (up to unique isomorphism) perfectoid space $\cM_{\infty}$ over $\C_p$ such that 
\[\cM_{\infty}\sim \varprojlim_K \cM_{K}.
\]
\end{thm}
The space $\cM_\infty$ represents the functor from complete affinoid \\
$(W(k)[\frac{1}{p}], W(k))$-algebras to $Sets$ which sends $(R,R^+)$ to the set of triples $(G, \rho, \alpha)$ where $(G, \rho) \in \cM_0(R,R^+)$ and 
\[\alpha: \Z_p^2 \rightarrow T(G)_\eta^{ad}(R,R^+)\]
is a morphism of $\Z_p$-modules such that for all points $x = \Spa(K,K^+) \in \Spa(R,R^+)$, the induced map 
\[\alpha(x): \Z_p^2 \rightarrow T(G)_\eta^{ad}(K,K^+)
\]
is an isomorphism, cf.\ \cite[Section 6.3]{SW}. Here $T(G)_\eta^{ad}$ denotes the generic fibre of the Tate module of $G$.  
We have a continuous right action of $\GL_2(\Q_p)$ on $\cM_\infty$. The action of $\GL_2(\Z_p)$ is easy to describe: An element $g \in \GL_2(\Z_p)$ acts as 
\[(G, \rho, \alpha) \mapsto (G, \rho, \alpha \circ g).\] 

We have a decomposition into perfectoid spaces
\[\cM_{\infty} \cong \bigsqcup_{i \in \Z} \cM_{\infty}^{(i)} \]
at infinite level and for each $i\in \Z$ the stabilizer $\mathrm{Stab}_{\GL_2(\Q_p)}(\cM^{(i)}_\infty)$ is given by the subgroup
\[G':=\{g \in \GL_2(\Q_p) | \ \det(g) \in \Z_p^*\}.\]

As the system $(K(m))_{m\geq 0}$ is cofinal in the system of compact open subgroups $K\subset \GL_2(\Q_p)$ we also have
\begin{equation}\label{sim1} \cM_{\infty}^{(i)} \sim \varprojlim \cM_{m}^{(i)}.
\end{equation}
for any $i \in \Z$. We denote by $p_m$ the map  
\begin{equation*}
p_m: \cM^{(0)}_{\infty}\rightarrow \cM^{(0)}_{m}
\end{equation*} implicit in (\ref{sim1}). 

\subsection{The quotient $\cM^{(0)}_{\infty}/B(\Z_p)$}
Consider the perfectoid space $\cM^{(0)}_{\infty}$. Let $B \subset \GL_2$ the Borel subgroup of upper triangular matrices. The group $B(\Z_p)$ acts continuously on $\cM^{(0)}_{\infty}$. In Section \ref{subsec: 3.4} we prove the following theorem.

\begin{thm}\label{LTquotient} The object $\cM^{(0)}_{\infty}/B(\Z_p) \in \mathcal{V}$ is a perfectoid space. It comes equipped with compatible maps 
\[ \cM^{(0)}_{\infty}/B(\Z_p) \rightarrow \cM^{(0)}_{m}/B(\Z/p^m\Z)\]
and satisfies
\[\cM^{(0)}_{\infty}/B(\Z_p) \sim \varprojlim_{m} \cM^{(0)}_{m}/B(\Z/p^m\Z).
\]
There is a cover $(U_i=\spa(R_i,R^+_i))_{i \in I}$ of $\cM^{(0)}_{\infty}$ by $B(\Z_p)$-invariant affinoid perfectoid spaces such that the image $\phi(U_i)$ under the projection map
\[\phi: \cM^{(0)}_{\infty} \rightarrow \cM^{(0)}_{\infty}/B(\Z_p)\] 
is affinoid perfectoid and given by $\spa(R_i^{B(\Z_p)},(R_i^+)^{B(\Z_p)})$.
\end{thm}

The strategy to prove the above theorem is the following: 

\begin{itemize}
	\item Using infinite level modular curves, we will first construct a perfectoid space $\cM_{B}$ with the property that
\[\cM_{B} \sim \varprojlim \cM^{(0)}_{m}/B(\Z/p^m\Z).
\]
By Proposition \ref{unique} the space $\cM_{B}$ is then unique up to unique isomorphism.
\item In Section \ref{subsec: 2.5} we then identify $\cM_{B}$ with $\cM^{(0)}_{\infty}/B(\Z_p)$.
\end{itemize}

\subsection{Background on infinite level modular curves}
The infinite-level modular curve $ \cX_{\Gamma(p^{\infty})}^* $ was constructed as a perfectoid space over $\Q^{\mathrm{cycl}}_p$ in \cite{torsion}. In this paper we will work with the base change to $\C_p$. Furthermore $\Gamma_0(p^m)$ denotes the subgroup of $\GL_2(\Z_p)$ consisting of matrices that are upper triangular $\mod p^m$, without the determinant condition in \cite{torsion}. This affects only connected components as we work over $\C_p$. In what follows, all other notation is as in \cite{torsion}, but the spaces are always implicitly base changed to $\C_p$.

The perfectoid space $ \cX_{\Gamma(p^{\infty})}^* $ comes equipped with a compatible family of maps $q_m: \cX^*_{\Gamma(p^{\infty})}\rightarrow \cX^*_{\Gamma(p^{m})}$ such that 
\[
\cX_{\Gamma(p^{\infty})}^* \sim \varprojlim \cX^*_{\Gamma(p^{m})}.
\]
The group $\GL_2(\Q_p)$ acts continuously on $\cX_{\Gamma(p^{\infty})}^*$ and there is a surjective $\GL_2(\Q_p)$-equivariant map, the Hodge-Tate period map
\[\pi_{\HT}:\cX^*_{\Gamma(p^{\infty})} \rightarrow \mathbb{P}_{\C_p}^1.\]
The space $\mathbb{P}^1_{\C_p}$ comes equipped with a right action of $\GL_2(\Q_p)$, which on $\C_p$-points is given by
\[[x_0:x_1]\left(\begin{array}{cc} a& b \\ c & d \end{array}\right)= [dx_0-bx_1: -cx_0 + ax_1],
 \]
i.e., it is the usual left action of $\left(\begin{array}{cc} a& b \\ c & d \end{array}\right)^{-1}$. Furthermore part of the tower $(\cX^*_{\Gamma_0(p^{m})})_m$ is perfectoid. More precisely let $0 \leq\epsilon < 1/2$. Then Scholze constructs open neighbourhoods 
\[\cX^*_{\Gamma_0(p^{m})}(\epsilon)_a \subset \cX^*_{\Gamma_0(p^{m})}
\] 
of the anticanonical locus. For $m$ large enough these are affinoid spaces. 
Roughly speaking $\cX^*_{\Gamma_0(p^{m})}(\epsilon)_a$ is the locus, where the subgroup is anticanonical and the Hasse invariant satisfies $|Ha|\geq |p|^{p^{-m}\epsilon}$.
By Corollary III.2.19 of \cite{torsion}, there exists a unique affinoid perfectoid space  
$\cX^*_{\Gamma_0(p^{\infty})}(\epsilon)_a$ such that
\begin{equation}\label{simlim2}
\cX^*_{\Gamma_0(p^{\infty})}(\epsilon)_a \approx \varprojlim \cX^*_{\Gamma_0(p^{m})}(\epsilon)_a.
\end{equation}

This is in fact the first step in showing that $\cX^*_{\Gamma(p^{\infty})}$ is perfectoid. From $\Gamma_0(p^\infty)$ one can pass to deeper level structures as follows. 
Firstly, in \cite[Proposition III.2.33]{torsion}, the fibre product
\[ \cX^*_{\Gamma_1(p^{m})\cap \Gamma_0(p^{\infty})}(\epsilon)_a := \cX^*_{\Gamma_1(p^{m})} \times_{\cX^*_{\Gamma_0(p^{m})}} \cX^*_{\Gamma_0(p^{\infty})}(\epsilon)_a
\]
is shown to be affinoid perfectoid. 
Passing from $\Gamma_1(p^m)$ to full level is easy: The map
\[\cX^*_{\Gamma(p^{m})}(\epsilon)_a \rightarrow \cX^*_{\Gamma_1(p^{m})}(\epsilon)_a
\]
is finite \'etale. Therefore the space
\begin{eqnarray*} \cX^*_{\Gamma(p^{m})\cap \Gamma_0(p^{\infty})}(\epsilon)_a &:=& \cX^*_{\Gamma(p^{m})}(\epsilon)_a \times_{\cX^*_{\Gamma_1(p^{m})}(\epsilon)_a} \cX^*_{\Gamma_1(p^{m})\cap \Gamma_0(p^{\infty})}(\epsilon)_a \\
&=& \cX^*_{\Gamma(p^{m})}(\epsilon)_a \times_{\cX^*_{\Gamma_0(p^{m})}(\epsilon)_a} \cX^*_{\Gamma_0(p^{\infty})}(\epsilon)_a
\end{eqnarray*}
is affinoid perfectoid. 

In the limit (cf.\ \cite[Theorem III.2.36]{torsion}) we get an affinoid perfectoid space $\cX^*_{\Gamma(p^{\infty})}(\epsilon)_a$ which satisfies
\begin{equation}\label{simlim1}
\cX^*_{\Gamma(p^{\infty})}(\epsilon)_a \approx \varprojlim \cX^*_{\Gamma(p^{m})}(\epsilon)_a
\end{equation}
and in particular
\[\cX^*_{\Gamma(p^{\infty})}(\epsilon)_a \approx \varprojlim \cX^*_{\Gamma(p^{m})\cap \Gamma_0(p^\infty)}(\epsilon)_a.
\]

\begin{prop}\label{open} The natural map
\[\phi: \cX^*_{\Gamma(p^{\infty})}(\epsilon)_a \rightarrow \cX^*_{\Gamma_0(p^{\infty})}(\epsilon)_a\]
is open.
\end{prop}
\begin{proof} Using the above relations (\ref{simlim2}) and (\ref{simlim1}) it suffices to show that the map
\[\cX^*_{\Gamma(p^{m})}\rightarrow \cX^*_{\Gamma_0(p^{m})}\]
is open. But this map is flat, therefore open. 
\end{proof}

We have a set-theoretic decomposition 
\begin{eqnarray*} \cX^*_{\Gamma(p^{\infty})} &=& \cX^{* \mathrm{ord}}_{\Gamma(p^{\infty})} \sqcup \cX^{\mathrm{ss}}_{\Gamma(p^{\infty})} \\
&=& \pi_{\HT}^{-1}(\mathbb{P}^1(\Q_p)) \sqcup \pi_{\HT}^{-1}(\Omega^2).
\end{eqnarray*}
Here $\Omega^2$ denotes the Drinfeld upper half plane $\Omega^2=\mathbb{P}^1\backslash \mathbb{P}^1(\Q_p)$.

The supersingular locus $\cX^{\mathrm{ss}}_{\Gamma(p^{\infty})}$ is contained in translates of the affinoid perfectoid space $\cX^*_{\Gamma(p^{\infty})}(\epsilon)_a$. 
Concretely, let $\gamma:= \left(\begin{array}{cc} p & \\ & p^{-1} \end{array}\right) \in \GL_2(\Q_p)$.
Then
\[ \cX^{\mathrm{ss}}_{\Gamma(p^{\infty})} \hookrightarrow \cX^*_{\Gamma(p^\infty)}\backslash \pi^{-1}_{\HT}([1:0]) = \bigcup_{i=0}^{\infty} \cX^*_{\Gamma(p^{\infty})}(\epsilon)_a \gamma^{i}. \]

The Lubin--Tate tower sits inside the infinite level modular curve. More precisely we have a decomposition of the supersingular locus  
\[\cX^{\mathrm{ss}}_{\Gamma(p^{\infty})} \cong \bigsqcup_{j\in J} \cM^{(0)}_{\infty}\]
into a finite disjoint union of copies of $\cM^{(0)}_\infty$.
The embeddings are equivariant for the $\GL_2(\Z_p)$-action. 

We fix once and for all an embedding $\iota:\cM^{(0)}_{\infty}\rightarrow \cX^{\mathrm{ss}}_{\Gamma(p^{\infty})}$
as well as compatible embeddings $\iota_m: \cM^{(0)}_{m}\rightarrow \cX^{\mathrm{ss}}_{\Gamma(p^{m})}$ such that for any $m\geq 0$ the following diagram commutes\footnote{Equivalently: Fix a supersingular elliptic curve $E/\overline{\mathbb{F}}_p$ and an isomorphism $E[p^\infty] \cong H$.}
\begin{equation}\label{lg}
\xymatrix{
 \cM^{(0)}_{\infty} \ar[d]^{p_m}\ar[r]^{\iota} &  \ar[d]^{q_m} \cX^*_{\Gamma(p^{\infty})}\\
\cM^{(0)}_{m} \ar[r]^{\iota_m} &  \cX^*_{\Gamma(p^{m})}\\
}.
\end{equation}

\subsection{Construction of $\cM_B$} \label{subsec: 3.4}

Define the underlying topological space of $\cM_B$ to be the quotient space $|\cM^{(0)}_{\infty}/B(\Z_p)|:=|\cM^{(0)}_{\infty}|/B(\Z_p)$.
Then 
\[|\cM^{(0)}_{\infty}/B(\Z_p)| =  \varprojlim_{m} |\cM^{(0)}_{m}/B(\Z/p^m\Z)|.
\]

In analogy with Definition III.3.5 of \cite{torsion} we make the following definition
\begin{defn} 
\begin{enumerate}
	\item A subset $V \subset |\cM^{(0)}_{\infty}/B(\Z_p)|$ is affinoid perfectoid if it is the preimage of some affinoid open 
	\[V_m = \Spa(R_m, R^+_m) \subset |\cM^{(0)}_{m}/B(\Z/p^m\Z)|
	\] 
	for all sufficiently large $m$, and if $(R_\infty, R^+_\infty)$ is an affinoid perfectoid $\C_p$-algebra, where $R^+_{\infty}$ is the $p$-adic completion of $\varinjlim_m R_m^+$, and $R_\infty:= R_\infty^+[p^{-1}]$.
\item A subset $V\subset |\cM^{(0)}_{\infty}/B(\Z_p)|$ is called perfectoid if it can be covered by affinoid perfectoid subsets.
\end{enumerate}
\end{defn}

Once we show that $|\cM^{(0)}_{\infty}/B(\Z_p)|$ is perfectoid in the sense of the previous definition, we may use the affinoid perfectoid algebras on the affinoid perfectoid pieces to define a structure sheaf on $|\cM^{(0)}_{\infty}/B(\Z_p)|$ turning it into the perfectoid space $\cM_B$ that we seek to construct.

In a first step we show that every point $x \in |\cM^{(0)}_{\infty}|$ has a nice $B(\Z_p)$-stable neighbourhood. For that we recall some explicit $B(\Z_p)$-invariant affinoid open subspaces of $\Omega^2$.

Let $\mathbb{D}^1 \subset \mathbb{P}^1$ be the closed unit disc embedded in $\mathbb{P}^1$ via $x \mapsto [x:1]$ in usual homogeneous coordinates.
For an integer $n \geq 1$ let $z_1,\dots,z_{p^{n}} \in \Z_p$ be a set of representatives of $\Z_p/p^n\Z_p$ and consider the rational subset $X_n \subset \mathbb{D}^1$ defined by 
\begin{eqnarray*} X_n(\C_p)  &=& \{ x \in \mathbb{D}^1(\C_p): |x-z_i| \geq p^{-n} \text{ for all } i = 1, \dots, p^n\} \\
 &=& \{ x \in \mathbb{D}^1(\C_p) : |x-z|\geq p^{-n} \text{ for all } z \in \Z_p \}.
\end{eqnarray*}
It is well known that $(X_n)_{n \in \mathbb{N}}$ is a cover of $\Omega^2\cap \mathbb{D}^1$. 
For any $n\geq 0$, the affinoid open $X_n \subset \mathbb{P}^1$ is stable under $B(\Z_p)$. Indeed let $g = \left(^a \ ^b_d\right) \in B(\Z_p)$ and $x \in X_n(\C_p)$, then
\[ |x-z|\geq p^{-n}, \ \forall \ z \in \Z_p \]
and therefore 
\begin{eqnarray*} |xg-z| &=& |(dx-b)/{a} -z|  \\
&=& |dx-b - az| = |x-(az+b)/d|\\
& \geq& p^{-n}. 
\end{eqnarray*}

For $i \geq 0$ the translate $\mathbb{D}^1\gamma^{i}$ is the disk of radius $p^{2i}$
\[ \mathbb{D}^1\gamma^{i}(\C_p) = \{[x:1] \in \mathbb{P}^1(\C_p) : |x|\leq p^{2i}\}.\]
The translates $X_n\gamma^{i}$ are invariant under the group
\[\gamma^{-i}B(\Z_p) \gamma^{i}= \left\{\left(\begin{array}{cc}a & b \\ & d \end{array}\right) \in B(\Q_p) : a, d \in \Z_p^*, b \in p^{-2i}\Z_p \right\}.
\]
Note that $B(\Z_p) \subset \gamma^{-i}B(\Z_p) \gamma^{i}$. Explicitly we have 
\[X_n\gamma^{i}(\C_p) =\{x \in \mathbb{D}^1\gamma^{i}(\C_p) : |x-z| \geq p^{2i-n} \ \forall z \in \Q_p \cap \mathbb{D}^1\gamma^{i}(\C_p) \}.\] 
By \cite{torsion} Theorem III.3.17(i), the preimage 
\[\pi_{\HT}^{-1}(\mathbb{D}^1) \subset \cX^{*}_{\Gamma(p^{\infty})}\]
under the Hodge-Tate period morphism is affinoid perfectoid\footnote{It coincides with the subset  
$\{[x_1:x_2] \in \mathbb{P}^1(\C_p) : |x_1| \leq |x_2|\}$ denoted by $\Fl_{\{2\}}$ in \cite{torsion}.}.

Furthermore there exists $0<\epsilon <\frac{1}{2}$ such that 
$\cX^*_{\Gamma(p^{\infty})}(\epsilon)_a \subset \pi_{\HT}^{-1}(\mathbb{D}^1)$. Note this is an inclusion of affinoid perfectoid spaces.
We fix such an $\epsilon$ and abbreviate 
\[Y:= \cX^*_{\Gamma(p^{\infty})}(\epsilon)_a.\]

\begin{prop} Let $x \in |\cM^{(0)}_{\infty}\cap Y\gamma^{i}|$ be a point. Then there exists an open neighbourhood $U$ of $x$ in $\cM_{\infty}^{(0)}$, such that
\begin{itemize}
	\item $U$ is affinoid perfectoid,
	\item $U \subset Y\gamma^{i}$ is a rational subset,
	\item $U$ is invariant under $\gamma^{-i}B(\Z_p)\gamma^i$,
	\item $U \approx \varprojlim U_m$ for affinoid open subsets $U_m \subset \cM^{(0)}_{m}$ and $m$ large enough.
\end{itemize}
\end{prop}
\begin{proof} 
We first show the assertion for all $x \in \left|\cM^{(0)}_{\infty} \cap Y \right|$. 
So let $x \in \left|\cM^{(0)}_{\infty} \cap Y\right|$ be arbitrary. Then there exists $n \in \mathbb{N}$ such that $|\pi_{\HT}|(x) \in |X_n|$. Fix such a $n$. The affinoid open $X_n \subset \mathbb{D}^1$ is a rational subset and therefore $\pi_{\HT}^{-1}(X_n)\subset \pi_{\HT}^{-1}(\mathbb{D}^1)$ is also a rational subset.
Note that $\pi_{\HT}^{-1}(X_n) \subset \cX^{\mathrm{ss}}_{\Gamma(p^{\infty})}$, so $\pi_{\HT}^{-1}(X_n)$ is a disjoint union of the affinoids $V_j$ defined as the intersection of $\pi_{\HT}^{-1}(X_n)$ and a copy of $\cM^{(0)}_{\infty}$. 
Let $j$ be such that $x \in V_j$. Still $V_j$ is rational in $\pi_{\HT}^{-1}(\mathbb{D}^1)$ and so 
\[U:= V_j \cap Y
\] 
is a rational subset of the affinoid perfectoid space $Y$.
In particular $U$ is affinoid perfectoid.

As we have seen above, the affinoid opens $X_n\subset \mathbb{P}^1$ are invariant under $B(\Z_p)$ and the Hodge-Tate period map $\pi_{\HT}$ is $\GL_2(\Q_p)$-equivariant, therefore $\pi_{\HT}^{-1}(X_n)$ is stable under $B(\Z_p)$. Each copy of $\cM^{(0)}_{\infty} \subset \cX^{\mathrm{ss}}_{\Gamma(p^{\infty})}$ is stable under $\GL_2(\Z_p)$ and $Y$ is also $B(\Z_p)$-invariant, which gives that $U$ is invariant under $B(\Z_p)$.
 
As $U \subset Y$ is rational, it follows from Lemma \ref{1.1} and Equation (\ref{simlim1}) that for $m$ large enough there exists $U_m \subset \cX^*_{\Gamma(p^{m})}(\epsilon)_a$ affinoid such that $U \approx \varprojlim U_m$.
As $U \subset \cM^{(0)}_{\infty}$, the commutativity of diagram (\ref{lg}) implies that $U_m \subset \cM^{(0)}_{m}$.

Now let $x \in |\cM^{(0)}_{\infty}|$ be arbitrary. Then there exists an integer $i\geq 0$ such that $x \in |Y\gamma^{i}|$ and an integer $n\geq 0$ such that $\pi_{\HT}(x) \in X_n\gamma^{i}$. The property of being affinoid perfectoid is stable under the action of $\GL_2(\Q_p)$, therefore
\[\pi^{-1}_{\HT}(\mathbb{D}^1\gamma^{i}),\pi^{-1}_{\HT}(X_n\gamma^{i}) \text{ and } Y\gamma^{i}  \]
are all affinoid perfectoid.
Define 
\[U:= \pi^{-1}_{\HT}(X_n\gamma^{i}) \cap Y\gamma^{i} \cap \cM^{(0)}_{\infty}.\] 
The same argument as above shows that this is affinoid perfectoid. It is stable under $\gamma^{-i}B(\Z_p)\gamma^i$ as $\gamma^{i}$ leaves each copy $\cM^{(0)}_{\infty} \subset \cX^{\mathrm{ss}}_{\Gamma(p^{\infty})}$ invariant. 

For the final assertion note that the property `$\approx$' for affinoids in $\cX^*_{\Gamma(p^{\infty})}$ is stable under $\GL_2(\Q_p)$ (see \cite{torsion} p.59)
therefore $U \approx \varprojlim U_m $ for affinoid opens $U_m \subset \cX^*_{\Gamma(p^{m})}$ and $m$ large enough, and we argue as above to deduce that $U_m \subset \cM^{(0)}_{m}$.
\end{proof}

\begin{defn} For $i\geq 0$ define $\mathcal{U}_i$ to be the collection of affinoid perfectoid open subsets $U\subset \cM^{(0)}_{\infty}$ such that 
\begin{enumerate}
	\item $U \subset Y\gamma^i$ is a rational subset and $\gamma^{-i}B(\Z_p)\gamma^i$-stable, 
	\item $U \approx \varprojlim_m U_m $ for affinoid open subsets $U_m \subset \cM^{(0)}_{m}$, $m$ large enough.
\end{enumerate}
\end{defn}
By the above proposition every point in $ |\cM^{(0)}_{\infty}|$ has a neighbourhood $U \in \mathcal{U}_i$ for some $i$.
We now construct the candidate quotients $U_B$ for any $U \in \mathcal{U}_i$.

\begin{prop} Let $U \in \mathcal{U}_0$. 
There exists a perfectoid space $U_B$ whose underlying topological space is homeomorphic to $|U|/B(\Z_p)$ and such that 
	\[
U_B \sim \varprojlim U_m/B(\Z/p^m\Z).
\]
\end{prop}

\begin{proof} By Proposition \ref{open}, the natural map 
\[\phi: Y=\cX^*_{\Gamma(p^{\infty})}(\epsilon)_a \rightarrow \cX^*_{\Gamma_0(p^{\infty})}(\epsilon)_a\]
of affinoid perfectoid spaces is open. Define 
\[U_B:= \phi (U).
\] 
This is a perfectoid space.

Let $(V^k)_{k\in I}$ be a cover of $U_B$ by rational subsets $V^k \subset \cX^*_{\Gamma_0(p^{\infty})}(\epsilon)_a$.
Then $V^k \approx \varprojlim V_m^k$ for affinoid subsets $V_m^k \subset \cX^*_{\Gamma_0(p^{m})}(\epsilon)_a$. 
Note 
\[ 
U^k:= \phi^{-1}(V^k)
\] 
is still rational in $\cX^*_{\Gamma(p^{\infty})}(\epsilon)_a$. 
Recall that at finite level we have a finite \'etale morphism 
\[\phi_m:\cX^*_{\Gamma(p^{m})}(\epsilon)_a \rightarrow \cX^*_{\Gamma_0(p^{m})}(\epsilon)_a,\]
which is Galois with Galois group $B(\Z/p^m \Z)$. 
Furthermore we have a commutative diagram
\begin{equation}
\xymatrix{
 \cX^*_{\Gamma(p^{\infty})}(\epsilon)_a\ar[r]^{\phi}\ar[d]^{q_m} &  \ar[d]^{q^0_m} \cX^*_{\Gamma_0(p^{\infty})}(\epsilon)_a\\
\cX^*_{\Gamma(p^{m})}(\epsilon)_a\ar[r]^{\phi_m} &  \cX^*_{\Gamma_0(p^{m})}(\epsilon)_a\\
}.
\end{equation}
Then $U^k = q_m^{-1}(\phi_m^{-1}(V_m^k))$ and we get a map $U^k_m:=\phi_m^{-1}(V_m^k) \rightarrow V^k_m$ of affinoid spaces with $U^k_m/B(\Z/p^m\Z)\cong V^k_m$. 
In particular $V^k \approx \varprojlim U_m^k/B(\Z/p^m\Z)$. This also implies that on topological spaces we get a homeomorphism 
\[|V^k| \cong |U^k|/B(\Z_p) \cong \varprojlim|U_m^k|/B(\Z/p^m\Z).\]  
\end{proof}
From the proof of the previous proposition we also see that after passing to rational subsets we may assume that $U$ as well as $U_B$ are both affinoid perfectoid.

\begin{prop} 
Let $i \geq 1$ and $U\in \mathcal{U}_i$. There exists a perfectoid space $U_B$ with underlying topological space $|U|/B(\Z_p)$ and such that
\[U_B\sim \varprojlim_m U_m/B(\Z/p^m\Z).
\] 
\end{prop}
\begin{proof} 

As $U$ is invariant under $\gamma^{-i}B(\Z_p)\gamma^i$, the affinoid perfectoid space 
\[W:= U\gamma^{-i} \subset Y
\]
is $B(\Z_p)$-invariant. As $U \approx \varprojlim U_m$ for $m$ large enough and as the property $\approx$ is stable under the action of $\GL_2(\Q_p)$, we have $W \approx \varprojlim W_m$ for $m$ large enough. 
Arguing as above we get a perfectoid space 
$\phi(U\gamma^{-i}) \sim \varprojlim W_m/B(\Z /p^m\Z)$
and after possibly shrinking $U$ we may assume that $\phi(U\gamma^{-i})$ is affinoid perfectoid and that 
\[\phi(U\gamma^{-i}) \approx \varprojlim W_m/B(\Z /p^m\Z).
\]
In particular, for $W_m/B(\Z/p^m \Z)= \Spa(S_m, S_m^+)$, the limit 
\[(\varinjlim_m S_m^+)^{\wedge}[p^{-1}]\]
is affinoid perfectoid. 

Define 
\[ H_i:= \gamma^i B(\Z_p) \gamma^{-i} = \left\{ \left(\begin{array}{cc} a & b\\ & d \end{array}\right) \in B(\Z_p) \ | \ b \in p^{2i}\Z_p\right\} \subset B(\Z_p)
\]
and let 
\[ H_{i,m} \subset B(\Z/p^m\Z)
\]
be the image of $H_i$ under the natural reduction map $B(\Z_p)\rightarrow B(\Z/p^m\Z)$.

Consider the tower $(W_m/H_{i,m})_{m \geq 2i}$ of affinoid adic spaces 
\[W_m/H_{i,m}= \spa(R_m,R_m^+).\] 
The natural maps
\[W_m/H_{i,m}\rightarrow W_m/B(\Z/p^m\Z)\]
are finite \'etale maps and the degree does not change for $m$ large enough as $B(\Z_p)/H_i \rightarrow B(\Z/p^m\Z)/H_{i,m}$ is a bijection. In fact the diagram
\[\xymatrix{
W_{m+1}/H_{i,m+1} \ar[d]\ar[r] & \ar[d] W_{m+1}/B(\Z/p^{m+1}\Z)\\
W_m/H_{i,m}\ar[r] & W_m/B(\Z/p^m\Z) \\
}\] 
is cartesian. 
The pullback 
\[ W/H_i:= W_{2i}/H_{i,2i}\times_{W_{2i}/B(\Z/p^{2i}\Z)}\phi(W) \rightarrow \phi(W)\]
is finite \'etale, therefore affinoid perfectoid say $W/H_i=\spa(R,R^+)$.

Furthermore (\cite[Remark 2.4.3]{huber})
\[W/H_i \approx \varprojlim W_m/H_{i,m} \]
and in particular
\[R \cong (\varinjlim R^+_m)^{\wedge}[p^{-1}],\  R^+\cong (\varinjlim R^+_m)^{\wedge}.\]
Define 
\[U_B:= \Spa(R, R^+).\]
We verify
\[U_B \approx \varprojlim_m U_m/B(\Z/p^m \Z)
\]  
by checking that the towers $(U_m/B(\Z/p^m\Z))_{m \geq 4i}$ and $(W_m/H_{i,m})_{m \geq 4i}$ are equivalent, i.e., that the systems agree in $(\cM^{(0)}_{0})_{\proet}$. 

For that note that we have the following commutative diagrams
\[\xymatrix{
\cM^{(0)}_{\infty} \ar[d]\ar[r]^{\gamma^{-i}} & \cM^{(0)}_{\infty} \ar[d]\\
\cM^{(0)}_{K({m})}\ar[r]^{\gamma^{-i}} & \cM^{(0)}_{\gamma^i K(m) \gamma^{-i}} \\
} \ \ \ 
 \xymatrix{
\cM^{(0)}_{\infty} \ar[d]\ar[r]^{\gamma^{i}} & \cM^{(0)}_{\infty} \ar[d]\\
\cM^{(0)}_{K({m})}\ar[r]^{\gamma^{i}} & \cM^{(0)}_{\gamma^{-i} K(m) \gamma^{i}} \\
}
\]

Let $m\geq 2i$, then we have inclusions 
\[ \gamma^{i}K(m+4i)\gamma^{-i} \subset K(m+2i) \subset \gamma^i K(m) \gamma^{-i} \subset K(m-2i) \text{ and}\]
\[ \gamma^{-i}K(m+4i)\gamma^{i} \subset K(m+2i) \subset \gamma^{-i} K(m) \gamma^{i} \subset K(m-2i).\]
The open subspace $W_{m-2i} \subset \cM^{(0)}_{K(m-2i)}$ is the image of $U_m$ under the composite  
\[\cM^{(0)}_{K(m)}\stackrel{\gamma^{-i}}{\rightarrow} \cM^{(0)}_{\gamma^i K(m) \gamma^{-i}}\rightarrow  \cM^{(0)}_{K(m-2i)}, \]
where the last map is the natural one from the inclusion $\gamma^i K(m) \gamma^{-i} \subset K(m-2i)$. We get an induced map 
\[f^U_W(m,i):U_{m}/B(\Z/p^{m}\Z) \rightarrow W_{m-2i}/H_{i,m-2i}.
\]
Similarly we get a map 
\[ f^W_U(m-2i,i) :W_{m-2i}/H_{i,m-2i} \rightarrow U_{m-4i}/B(\Z/p^{m-4i}\Z).\] The composite 
\[f^W_U(m-2i,i)\circ f^U_W(m,i)\]
is the natural projection $U_{m}/B(\Z/p^{m}\Z) \rightarrow U_{m-4i}/B(\Z/p^{m-4i}\Z)$. 
Moreover
\[f^U_W(m-2i,i) \circ f^W_U(m,i)\]
agrees with the natural projection map $W_m/H_{i,m} \rightarrow W_{m-4i}/H_{i,m-4i}$. Therefore the pro-{\'e}tale systems
$(U_m/B(\Z/p^m\Z))_{m \geq 4i}$ and $(W_m/H_{i,m})_{m \geq 4i}$ 
are isomorphic, which is what we wanted to show.
\end{proof}

To summarize, we have proved the following theorem.
\begin{thm}\label{cover}
\begin{enumerate}
	\item There exists a unique (up to unique isomorphism) perfectoid space 
\[\cM_B \sim \varprojlim \cM^{(0)}_{m}/B(\Z/p^m\Z)\]
and a natural surjective map $\varphi:\cM^{(0)}_\infty \rightarrow \cM_B$ such that the diagram 
 \[\xymatrix{
\cM^{(0)}_{\infty} \ar[d]\ar[r]^{\varphi} & \cM_B \ar[d]\\
\cM^{(0)}_{m}\ar[r] & \cM^{(0)}_{m}/B(\Z/p^m \Z) \\
}
\]
commutes.  
\item There is a cover $\mathcal{U}_{\LT}=(U^j)_{j \in J}$ of $\cM^{(0)}_{\infty}$ of affinoid perfectoid spaces $U^j \approx \varprojlim U_m^j \in \bigcup_{i\geq 0}\mathcal{U}_i$ such that the image of $\varphi(U^j)$ under $\varphi:\cM^{(0)}_{\infty}\rightarrow \cM_B$
is affinoid perfectoid and satisfies 
\[ \varphi(U^j) \approx \varprojlim U_m^j/B(\Z/p^m\Z).
\]
\end{enumerate}
\end{thm}

\subsection{The quotient $\cM^{(0)}_{\infty}/B(\Z_p)$}\label{subsec: 2.5}

Finally we argue that the structure sheaf of $\cM_B$ is the expected one, using the pro\'etale site $ (\cM^{(0)}_0)_{\proet}$. Define $B_m:=B(\Z/p^m\Z)$.
First note that 
\[\cM_B \in (\cM^{(0)}_0)_{\proet}\]
as 
\[\cM_m^{(0)}/B(\Z/p^m\Z)\rightarrow \cM_{m-1}^{(0)}/B(\Z/p^{m-1}\Z)
\]
is finite \'etale surjective and $\cM_m^{(0)}/B(\Z/p^m\Z)$ is \'etale over $\cM^{(0)}_0$. 

\begin{prop} Let $U=\spa(R,R^+) \in \mathcal{U}_{\LT}$ 
with
$U \approx \varprojlim U_m=\spa(R_m, R^+_m)$. Then the natural map of affinoid algebras
\[ \left(S:=S^+\left[\frac{1}{p}\right], S^+:=\left(\varinjlim {R^+_m}^{B_m}\right)^{\wedge}\right) \rightarrow \left(R^{B(\Z_p)}, {R^+}^{B(Z_p)}\right)
\]
is an isomorphism. 
\end{prop}
\begin{proof} We have seen that $(S=\cO_{\cM_B}(\varphi(U)), S^+=\cO^+_{\cM_B}(\varphi(U)))$ is an affinoid perfectoid algebra. Consider the completed structure sheaf $\widehat{\cO}_{\cM^{(0)}_{0}}$ and the integral completed structure sheaf $\widehat{\cO}^+_{\cM^{(0)}_{0}}$ as defined in \cite[Def.\ 4.1]{scholzepHT}. These are sheaves on $(\cM^{(0)}_{0})_{\proet}$. Now $U, \varphi(U) \in (\cM^{(0)}_{0})_{\proet}$ are both affinoid perfectoid therefore by Lemma 4.10 (iii) of \cite{scholzepHT} 
\[S=\widehat{\cO}_{\cM^{(0)}_{0}}(\varphi(U)) , \ S^+=\widehat{\cO}^+_{\cM^{(0)}_{0}}(\varphi(U)).\]
Furthermore $U\rightarrow \varphi(U)$ is a covering so 
\[S=\widehat{\cO}_{\cM^{(0)}_{0}}(\varphi(U)) = eq \left(\widehat{\cO}_{\cM^{(0)}_{0}}(U)\rightrightarrows \widehat{\cO}_{\cM^{(0)}_{0}}(U\times_{\varphi(U)}U)\right)\]
and analogously for $S^+$. 
Now $\widehat{\cO}_{\cM^{(0)}_{0}}(U)=R$ and $\widehat{\cO}^+_{\cM^{(0)}_{0}}(U) = R^+$, again by Lemma 4.10 (iii) of \cite{scholzepHT}. Furthermore  
\[ U\times_{\varphi(U)}U  \approx \varprojlim_m U_m \times_{U_m/B(\Z/p^m\Z)} U_m \cong \varprojlim U_m \times B_m \approx U \times \underline{B(\Z_p)},\]
where the isomorphism in the middle is implied by Lemma \ref{free}. Corollary 6.6 of \cite{scholzepHT} then implies that 
\[
\widehat{\cO}_{\cM^{(0)}_{\LT,0}}(U\times_{\varphi(U)}U) = \operatorname{Map}_{\mathrm{cont}}(B(\Z_p),R)
\] 
\[
\widehat{\cO}^+_{\cM^{(0)}_{\LT,0}}(U\times_{\varphi(U)}U) = \operatorname{Map}_{\mathrm{cont}}(B(\Z_p),R^+)
\] 
and the two maps in each of the equalizers above are given as 
$r \mapsto (g \mapsto r)$ and $r \mapsto (g \mapsto gr)$. Therefore $S \cong R^{B(\Z_p)}$ and $S^+ \cong {R^+}^{B(\Z_p)}$.  
\end{proof}

In particular we get the following corollary, which finishes the proof of Theorem \ref{LTquotient}.
\begin{corollary} We have an isomorphism in the category $\mathcal{V}$
\[ \cM_B \cong \cM^{(0)}_{\infty}/B(\Z_p).
\]
In particular, $\cM^{(0)}_{\infty}/B(\Z_p)$ is a perfectoid space.
\end{corollary}

\subsection{The quotient $\cM_{\infty}/B(\Q_p)$}\label{subsec:3.6}
We now use the above results and the fact that the Gross--Hopkins period morphism $\pi_{\mathrm{GH},0}$ at level zero has local sections to prove the following theorem.
\begin{thm}\label{secmainthm} The object 
\[\cM_{\infty}/B(\Q_p)\]
of $\mathcal{V}$ is a perfectoid space. 
\end{thm}
\begin{proof} Define the subgroup 
 \[B':=\{b \in B(\Q_p)| \det(b) \in \Z_p^*\}\subset B(\Q_p).\] 
It is the kernel of the map $val\circ \det: B(\Q_p) \rightarrow \Z$ and $B(\Q_p) \cong B' \rtimes \Z$.
We have 
\[\cM_{\infty} \cong \cM^{(0)}_{\infty} \times \underline{\Z} \cong \bigsqcup_{i \in \Z} \cM^{(0)}_{\infty}
\]
and 
\[B'= \mathrm{Stab}_{B(\Q_p)}(\cM^{(0)}_{\infty}).\]
This implies that in $\mathcal{V}$ we have an isomorphism 
\[\cM_{\infty}/B(\Q_p) \cong \cM^{(0)}_{\infty}/B',
\]
so it suffices to show that the latter object is a perfectoid space.

The Gross--Hopkins period map 
\[\pi_{\mathrm{GH},0}:\cM^{(0)}_{0}\rightarrow \mathbb{P}_{\C_p}^1\]
has local sections (cf.\ Lemma 6.1.4 in \cite{SW}). 

Let $U\subset \cM^{(0)}_{0}$ be open affinoid such that $\pi_{\mathrm{GH},0}|_U$ is an isomorphism onto its image. As before let $p_0:\cM_{\infty}^{(0)}\rightarrow \cM^{(0)}_{0}$ be the natural map and consider $p_0^{-1}(U) \subset \cM^{(0)}_{\infty}$. It has an action of $B(\Z_p)$.
Consider the object $p_0^{-1}(U)\times^{B(\Z_p)} B' \in \mathcal{V}$ as defined in Section \ref{subsec:quotients}. 
By Lemma \ref{immersion} we get a natural map  
\[\iota_U: p_0^{-1}(U)\times^{B(\Z_p)} B' \rightarrow \cM^{(0)}_{\infty}. 
\] 
in $\mathcal{V}$. As the geometric fibres of $\pi_{\GH,0}$ are isomorphic to $G'/\GL_2(\Z_p) \cong B'/B(\Z_p)$ and  $\pi_{\mathrm{GH},0}|_U$ is an isomorphism onto its image, $|\iota_U|$ is an injection, therefore $\iota_U$ is an open embedding (again by Lemma \ref{immersion}).  
But then  
\[(p_0^{-1}(U)\times^{B(\Z_p)} B')/B' \hookrightarrow \cM^{(0)}_{\infty}/B' 
\] 
is also an open embedding.

Using Proposition \ref{nh} we see that
\[(p_0^{-1}(U)\times^{B(\Z_p)} B')/ B' \cong p_0^{-1}(U)/B(\Z_p)
\]
is a perfectoid space. This finishes the proof as the $(p_0^{-1}(U)\times^{B(\Z_p)} B')/ B'$ cover $\cM^{(0)}_{\infty}/B' $.
\end{proof}

The Gross--Hopkins period morphism $\pi_{\GH}:\cM_\infty \rightarrow \mathbb{P}^1_{\C_p}$ factors through the quotient $\cM_\infty/B(\Q_p)$. We denote the induced map by 
\[\overline{\pi}_{\GH}:\cM_\infty/B(\Q_p)\rightarrow \mathbb{P}^1_{\C_p}.\]

\begin{prop}\label{qc} The map $\overline{\pi}_{\GH}: \cM_{\infty}/B(\Q_p) \rightarrow \mathbb{P}_{\C_p}^1$ is quasicompact.
\end{prop}
\begin{proof} Again we use that the Gross--Hopkins period map at level zero has local sections. So we can cover $\mathbb{P}_{\C_p}^1$ by affinoid opens $U$ that admit a section to some affinoid $V \subset \cM^{(0)}_{0}$. Now for any such $U$, 
\[\overline{\pi}_{\GH}^{-1}(U) \cong (p_0^{-1}(V)\times^{B(\Z_p)} B')/ B'\]
as $B'/B(\Z_p) \cong G'/\GL_2(\Z_p)$. So it suffices to show that each $(p_0^{-1}(V)\times^{B(\Z_p)} B')/ B'$ is quasi-compact. By \cite[Prop.\ 3.12]{scholzepHT}, $p_0^{-1}(V)$ is quasicompact. 
But now 
\[(p_0^{-1}(V)\times^{B(\Z_p)} B')/ B' \cong p_0^{-1}(V)/{B(\Z_p)}\]
and the latter is quasi-compact as $p_0^{-1}(V)$ is. 
\end{proof}

\begin{rem} From Proposition \ref{univspec} we see that the space $\cM_\infty$ is partially proper, so $\cM_\infty/B(\Q_p)$ is partially proper and by the last proposition it is quasi-compact, therefore $\cM_\infty/B(\Q_p)$ may be called proper.
\end{rem}

We finish this section by proving that the fibres of $\overline{\pi}_{\GH}$ are affinoid perfectoid. For that we start with the following result on rank one points. 

\begin{prop}\label{rk1} Let $x \in \mathbb{P}^1_{\C_p}$ be a rank one point. Then there exists an affinoid open neighbourhood $x \in U$ such that the preimage $\overline{\pi}^{-1}_{\GH}(U)$ is affinoid perfectoid. 
\end{prop}
\begin{proof} Let $U'\subset \mathbb{P}^1$ be an affinoid neighbourhood of $x$ and $V \subset \cM^{(0)}_0$ an affinoid open such that the map ${\pi_{\GH,0}}|_{V}:V \rightarrow U'$ is an isomorphism. There exists a cover of $\cM^{(0)}_0$ by increasing open affinoid spaces $\spa(B_i,B_i^+)$ such that $p_0^{-1}(\spa(B_i,B_i^+))$ is affinoid perfectoid (cf.\ \cite[Section 2.10]{weinstein}). 
As $V$ is quasi-compact we can find $i \in I$ such that $V\subset \spa(B_i,B_i^+)$ and we can cover $V$ by open affinoids that are rational in $\spa(B_i,B_i^+)$ so we may assume $V \subset  \spa(B_i,B_i^+)$ is rational, in particular that $p_0^{-1}(V)$ is affinoid perfectoid.

Using the notation of Section 2.4, $p_0^{-1}(V) \subset Y\gamma^i$ for some $i$, therefore we can find a rational subset $W\subset Y\gamma^i$ that is contained in $p_0^{-1}(V)$ with $x \in (\pi_{\GH,0}\circ p_0)(W)$ and such that $W \in \mathcal{U}_{\LT}$. In particular there exists $m \geq 0$ and an affinoid open $W_m\subset \cM_{m}^{(0)}$ 
with $W = W_m \times_{\cM_{m}^{(0)}}\cM_{\infty}^{(0)}$. Let $x_m \in W_m$ be a lift of $x$ and let $H \subset B(\Z/p^m\Z)$ be the stabilizer of $x_m$. By Lemma \ref{nbhd} below there exists a rational subset $V_m \subset W_m$ with $x_m \in V_m$, $V_m \cdot H = V_m$ and $V_m \cdot b \cap V_m = \emptyset $ for all $1\neq b \in H\backslash B(\Z/p^m\Z)$. The affinoid 
\[ U_m:= \bigsqcup_{b \in H\backslash B(\Z/p^m\Z)} V_m \cdot b
\]
now comes from level zero: $U_m = p_{m,0}^{-1}(U_0)$ for $U_0:=p_{m,0}(U_m)$ and 
\[U:=\pi_{\GH,0}(U_0) 
\]
has the property that $\pi_{\GH,0}:U_0 \rightarrow U$ is an isomorphism, $x \in U$, $p_0^{-1}(U_0) \in \mathcal{U}_{\LT}$ and so
\[\overline{\pi}_{\GH}^{-1}(U) = p_0^{-1}(U_0)/B(\Z_p)
\]
is affinoid perfectoid.
\end{proof}
\begin{lem}\label{nbhd} Let $X$ be a separated rigid analytic space, $G$ finite group acting on $X$, and let $x \in X$ be a point of rank $1$. Let $H:=\mathrm{Stab}_G(x)$. Then there exists an open affinoid neighbourhood $U$ of $x$ such that
\begin{itemize}
	\item $H\cdot U = U$
	\item $ Ug \cap U = \emptyset $ for all $1 \neq g \in H\backslash G$. 
\end{itemize}
\end{lem} 
\begin{proof} As $x$ is of rank $1$, we have 
\[x = \bigcap_{x \in V} V \]
where the intersection runs over all affinoid opens $V$ containing $x$. Likewise $xg = \bigcap Vg$ for any $g \in G$.
If $g \in G$ is such that $x\cdot g \neq x$ then 
\[\bigcap (Vg \cap V) = \emptyset\]
but the $Vg \cap V$ are all quasicompact as they are affinoid as $X$ is separated. This implies that there exists an affinoid open $V\subset X$ containing $x$ and s.th.\ $Vg \cap V = \emptyset$ for all $g \in H\backslash G$. Fix such a $V$ and let $U:= \bigcap_{h \in H} Vh$. This is affinoid as $X$ is separated and has the properties we want. 
\end{proof}

\begin{defn} \label{lft}
\begin{enumerate}
	\item Let $(S,S^+) \rightarrow (R,R^+)$ be a morphism of Huber pairs. We say $(R,R^+)$ is of \emph{lightly finite type} over $(S,S^+)$ if $R^+\subset R$ can be generated by finitely many elements as an open and integrally closed $S^+$-subalgebra, i.e., if there exists a finite subset $E \subset R^+$ such that $R^+$ is the integral closure of $S^+[E \cup R^{\circ\circ}]$ in $R$.\footnote{This is similar to Huber's definition of ${^+}$weakly finite type, but he furthermore demands that the morphism is of topologically finite type.}
\item We say that a morphism $f:X\rightarrow Y$ of analytic adic spaces is \emph{locally of lightly finite type} if for any $x \in X$ there exist open affinoid subspaces $U, V$ of $X,Y$ such that $x \in U$, $f(U)\subset V$, and the induced morphism 
\[(\cO_Y(V), \cO^+_Y(V))\rightarrow (\cO_X(U), \cO^+_X(U))\]
is of lightly finite type.
\end{enumerate}
\end{defn}
\begin{rem} 
\begin{enumerate}
	\item It then follows that for any pair $U, V$ as in part (2) of the previous definition, the induced morphism $U \rightarrow V$ is locally of lightly finite type, as the property passes to rational subsets.  
\item Note furthermore, that any affinoid rigid space $X=\spa(R,R^\circ)$ over $\C_p$ is of lightly finite type.
\item The reason for introducing the notion ``of lightly finite type'' is the next lemma. In the proof there we arrive at the following situation: We are given a Huber ring $A$ and two rings $A^+ \subseteq (A^+)'$ of integral elements and we would like to know that natural morphism $\iota: \spa(A,(A^+)') \rightarrow \spa(A,A^+)$ is an open immersion. If $(A,A^+) \rightarrow (A, (A^+)')$ is of lightly finite type, then this is certainly true: $\iota$ is an isomorphism onto the rational subset of $\spa(A,A^+)$ defined by $|f_1|,\dots, |f_n| \leq 1$, where $\{f_1,\dots, f_n\}\subset (A^+)'$ generates $(A^+)'$ over $A$ in the sense of Definition \ref{lft}(1).
\end{enumerate}
	\end{rem}
	
\begin{lem} Assume  that $X$ is perfectoid and $X \rightarrow \spa(K,K^+)$ is partially proper and locally of lightly finite type. Then for any perfectoid $(K,K^+)$-algebra $(S,S^+)$ 
\[X(S,S^+) = X(S,S^\circ).\]
\end{lem}
\begin{proof}
Assume we have a morphism $f^\circ: Y^\circ:=\spa(S,S^\circ) \rightarrow X$ over $\spa(K,K^+)$. We extend it to a morphism 
$f: Y:=\spa(S,S^+)\rightarrow X$ over $\spa(K,K^+)$ such that the diagram 
\[\xymatrix{ & \spa(S,S^+) \ar[d]^f \\
\spa(S,S^\circ) \ar@{^{(}->}[ur]^j \ar[r]^{f^\circ} & X
}
\]
commutes. For that we define $f$ as a morphism in $\mathcal{V}$ as follows. We first define a map $|f|$ of underlying topological spaces. So let $y \in |\spa(S,S^+)|$ be a point. Then there is a unique morphism 
$p_y: \spa(k(y),k(y)^+) \rightarrow Y$ such that $y$ is the image of the unique closed point in $\spa(k(y),k(y)^+)$. Note that 
\[p_y(\spa(k(y), k(y)^\circ)) \subset Y^\circ,\]
 so we get a morphism
\[f^\circ \circ p_y|_{\spa(k(y),k(y)^\circ)}: \spa(k(y),k(y)^\circ) \rightarrow X.\]
As $X$ is partially proper over $\spa(K,K^+)$, this extends uniquely to a morphism
\[f_y:\spa(k(y),k(y)^+) \rightarrow X.\] 
Define $|f|(y)$ as the image of $f_y$ of the unique closed point in $\spa(k(y),k(y)^+)$.

To show that this defines a continuous map let $U:=\spa(R,R^+)\subset X$ be open affinoid. After possibly shrinking $U$ we may assume that the morphism $U\rightarrow \spa(K,K^+)$ is of lightly finite type. The preimage $(f^\circ)^{-1}(U)$ in $Y^\circ$ is open, so by passing to rational subsets we may assume WLOG that we have a diagram over $\spa(K,K^+)$
\[\xymatrix{
& Y  \ar[dr]^{|f|}& \\
Y^\circ \ar@{^{(}->}[ur]^j \ar[r]^{f^\circ} & U \ar@{^{(}->}[r] & X
}
\]
where $Y^\circ =  \spa(S,S^\circ)$ and all arrows except $|f|$ are maps of adic spaces. In particular, we have a morphism 
\[f^{\circ,\sharp}: (R,R^+)\rightarrow (S,S^\circ).\]
Define $(S^+)'$ as the integral closure of $S^+ \cup f^{\circ,\sharp}(R^+)$ in $S$. As $\spa(R,R^+) \rightarrow \spa(K,K^+)$ is of lightly finite type, there exists a finite set $E \subset R^+$ such that $R^+$ is the integral closure of $K^+[E\cup R^{\circ\circ}]$ in $R$. But now $(S^+)'$ agrees with the integral closure of $S^+ \cup f^{\circ,\sharp}(E)$ in $S$, in particular, $\spa(S,(S^+)') \subseteq \spa(S,S^+)=Y$ is open. By construction if $y \in Y$ such that $|f|(y) \in U$, then $y \in \spa(S,(S^+)')$, i.e., $|f|^{-1}(U)=\spa(S,(S^+)')$ is open in $Y$. 

Constructing the map of structure sheaves is now easy. Note that $j_* \mathcal{O}_{Y^\circ} = \cO_Y$ and $f^\circ$ gives a morphism $ \cO_X \rightarrow f^{\circ}_*\cO_{Y^\circ}$. But 
\[f^{\circ}_*\cO_{Y^\circ} \cong f_* j_* \cO_{Y^\circ} \cong f_*\cO_Y,\]
so we get a natural map
\[f^\sharp:\cO_X \rightarrow f_*\cO_Y.\]
Finally one checks that the morphism $(|f|,f^\sharp)$ is compatible with the valuations on the stalks, therefore indeed defines a morphism in $\mathcal{V}$.
\end{proof}

\begin{lem} Consider a cartesian diagram 
 \[\xymatrix{
X^\circ \ar[d]\ar[r] & X \ar[d]^f\\
\spa(K,K^\circ) \ar[r] & \spa(K,K^+) \\
}
\]
where $K$ is a non-archimedean field, $K^+\subset K$ is an open and bounded valuation subring, $X$ is a perfectoid space and $f$ is partially proper and locally of lightly finite type. Assume that $X^\circ= \spa(R,R^+)$ is affinoid perfectoid. Then $X$ is affinoid perfectoid.
\end{lem}
\begin{proof} We claim that $X = \spa(R,{(R^+)}')$, where $(R^+)'$ is the integral closure of $K^+ + R^{\circ \circ}$ in $R$. To see this let $(S,S^+)$ be a perfectoid $(K,K^+)$-algebra. Then by the previous lemma
\[X(S,S^+) = X(S,S^\circ) = X^\circ(S,S^\circ) = \mathrm{Hom}_{\cont}(R,S).\]
On the other hand 
\[\left(\spa(R,{(R^+)}')\right)(S,S^+) = \mathrm{Hom}_{\cont}(R,S,{(R^+)}' \rightarrow S^+)\]
is given by continuous $K$-algebra homomorphisms $R\rightarrow S$ that send ${(R^+)}'$ into $S^+$. But any continuous $K$-algebra homomorphism $R \rightarrow S$ has this property as the image of the topologically nilpotent elements $R^{\circ \circ}$ lands in $S^{\circ \circ} \subset S^+$, $K^+$ is mapped to $S^+$ and $S^+$ is integrally closed. 
\end{proof}

The map $\cM_\infty/B(\Q_p) \rightarrow \spa(\C_p, \mathcal{O}_{\C_p})$ is locally of lightly finite type as locally $\cM_\infty/B(\Q_p)$ is pro-finite \'etale over a rigid space. 
Therefore $\overline{\pi}_{\GH}$ is locally of lightly finite type as well.
We finish this section studying the fibres of $\overline{\pi}_{\GH}$.

Let $X$ be a perfectoid space, $Y$ an analytic adic space, $f:X\rightarrow Y$ a morphism and $y=\spa(K,K^+) \in Y$ be a point. Define the fibre $f^{-1}(y)$ as the space
\[f^{-1}(y) := \bigcap_{\stackrel{U \text{q.c.\ open}:}{ |U| \supset |f|^{-1}(y)}} U\]
This is naturally a perfectoid space. One can check this locally on $Y$ and reduce to the situation where $X$ is affinoid perfectoid and $Y$ is affinoid Tate. In this situation the fibre $f^{-1}(y)$ is affinoid perfectoid as whenever we have a direct limit $\varinjlim R_i$ of perfectoid algebras $(R_i, R_i^+)$ the uniform completion $R:= (\varinjlim R_i)^\wedge$ is again perfectoid.

Now let $x =\spa(K,K^+)$ be a point of $\mathbb{P}^1_{\C_p}$, and consider the fibre $\overline{\pi}_{\GH}^{-1}(x)$, 
\[\overline{\pi}_{\GH}^{-1}(x) = \bigcap_{\stackrel{U \text{q.c.\ open}:}{ |U| \supset |\overline{\pi}_{\GH}|^{-1}(x)}} U.\]
Note that we have an equality of uniform adic spaces 
\[x = \bigcap_{\stackrel{V \text{q.c.\ open}:}{ V \ni x}} V\]
and therefore the diagram
\[\xymatrix{
\overline{\pi}_{\GH}^{-1}(x) \ar[d]^h\ar[r]^f & \cM_\infty/B(\Q_p) \ar[d]\\
x=\spa(K,K^+) \ar[r] & \mathbb{P}^1_{\C_p} \\
}
\]
is cartesian in the category of uniform adic spaces. If $U=\spa(S,S^+)= \overline{\pi}_{\GH}^{-1}(V) \subset \cM_{\infty}/B(\Q_p)$ for an affinoid open $V=\spa(R,R^+) \subset \mathbb{P}^1_{\C_p}$, then $f^{-1}(U) = \spa(T,T^+)$ where $T$ is the uniform completion of $S\widehat{\otimes}_R K$ and $T^+$ is generated by $S^+ \cup K^+$. Therefore $h$ is locally of lightly finite type (as $\overline{\pi}_{\GH}$ is).  

We can now show that the fibres of $\overline{\pi}_{\GH}$ are all affinoid perfectoid.
\begin{prop}\label{affinoidfibres}
Let $x =\spa(K,K^+)$ be a point of $\mathbb{P}^1_{\C_p}$, then the fibre
\[\overline{\pi}_{\GH}^{-1}(x)= \bigcap_{\stackrel{U \text{q.c.\ open}:}{ |U| \supset |\overline{\pi}_{\GH}|^{-1}(x)}} U\]
is affinoid perfectoid. 
\end{prop}
\begin{proof} For points of rank one this follows from Proposition \ref{rk1}. For points $x =\spa(K,K^+)$ of higher rank the claim follows from the previous lemma and the fact that 
\[\xymatrix{
\overline{\pi}_{\GH}^{-1}(x_0) \ar[d]\ar[r] & \overline{\pi}_{\GH}^{-1}(x) \ar[d]^h\\
x_0=\spa(K,K^\circ) \ar[r] & x \\
}
\]  
is cartesian. 
\end{proof}

\section{Cohomological Consequences}
Let $F/\Q_p$ be a finite extension, $n\geq 1 $ an integer and let $D$ be the division algebra with center $F$ and invariant $1/n$. In \cite{scholzeLT}, Scholze constructs a functor 
\begin{eqnarray*}
\left\{\begin{aligned} &\text{smooth admissible}\\ & \mathbb{F}_p\text{-representations } \\ & \text{of} \GL_n(F)\end{aligned} \right\}
 &\longrightarrow& \left\{D^*\text{-equivariant sheaves on } (\mathbb{P}^{n-1})_{\et} \right\}\\
\pi &\longmapsto &\cF_\pi
\end{eqnarray*}
and shows that the cohomology groups 
\[ H^i_{\et}(\mathbb{P}^{n-1}_{\C_p},\mathcal{F}_\pi)\] 
come equipped with an action of $\Gal(\overline{F}/F)$, are admissible $D^*$-representations and vanish for all $i > 2(n-1)$.
In particular this provides evidence for the existence of a mod $p$ local Langlands correspondence and a mod $p$ Jacquet--Langlands correspondence.

Here, we specialize to the case of $F=\Q_p$ and $n=2$ and we abbreviate
\[\cS^i(\pi):=H^i_{\et}(\mathbb{P}^1_{\C_p},\mathcal{F}_\pi).\]

Let $q=p^k$, for some $k \geq 1$. For a pair $\chi=(\chi_1, \chi_2)$ of characters $\Q_p^* \rightarrow \mathbb{F}^*_q$ we write $\Ind(\chi)$ or $\Ind(\chi_1,\chi_2)$ for the smooth parabolic induction $\Ind_{B(\Q_p)}^{\GL_2(\Q_p)}(\chi)$, so
\[\Ind_{B(\Q_p)}^{\GL_2(\Q_p)}(\chi):= \left\{f:\GL_2(\Q_p) \rightarrow {\mathbb{F}}_q	 | f \text{ cont.\ and } f(bx)=\chi(b)f(x) \ \forall b \in B(\Q_p)\right\}
.\]
This smooth admissible representation is irreducible unless $\chi_1=\chi_2$. When $\chi_1=\chi_2$ we have a short exact sequence 
\[0 \rightarrow \chi_1 \rightarrow \Ind(\chi_1,\chi_1)\rightarrow \mathrm{St}\otimes \chi_1 \rightarrow 0\]
where $\mathrm{St}$ denotes the Steinberg representation, which is irreducible.  
\smallskip\\

In this section we show that for $\pi \cong \Ind(\chi_1,\chi_2)$ with $\chi_1\neq \chi_2$ and for $\pi\cong \mathrm{St}\otimes \mu$ any twist of the Steinberg representation the cohomology $\cS^i(\pi)$ is concentrated in degree one. The vanishing of $\cS^0(\pi)$ is easy. 

\begin{prop} 
\begin{enumerate}
	\item Let $\chi_i: \Q_p^* \rightarrow {\mathbb{F}}^*_q, i=1,2$ be two distinct smooth characters. Then 
 \[\cS^0(\Ind(\chi_1,\chi_2)) = 0.\]
\item Let $\pi:= \mathrm{St}\otimes \mu$ be a twist of the Steinberg representation by a character $\mu: \Q_p^* \rightarrow {\mathbb{F}}^*_q$. Then 
\[\cS^0(\mathrm{St}\otimes \mu) = 0.\]
\end{enumerate}
\end{prop}
\begin{proof} Note that $\mathrm{Ind}(\chi_1,\chi_2)^{\SL_2(\Q_p)}=0$ as well as $(\mathrm{St}\otimes \mu)^{\SL_2(\Q_p)}=0$. The claim now follows from \cite[Prop.\ 4.7]{scholzeLT}, where it is proved in particular that for any admissible $\mathbb{F}_p[\GL_2(\Q_p)]$-module $V$ the natural map
\[H^0_{\et}(\mathbb{P}^1_{\C_p}, \mathcal{F}_{V^{\SL_2(\Q_p)}})\hookrightarrow H^0_{\et}(\mathbb{P}^1_{\C_p}, \mathcal{F}_{V})\]
is an isomorphism. 
\end{proof}

To show vanishing of cohomology in degree two we use the quotient constructed above. Again let $\chi_i: \Q_p^* \rightarrow \mathbb{F}_q^*, i=1,2$, be two smooth characters.
Let $\cF_{\chi}$ be the sheaf on $(\cM_{\infty}/B(\Q_p))_{\et}$ defined as
\[\cF_{\chi}(U) = \Map_{\cont,B(\Q_p)}(|U\times_{\cM_{\infty}/B(\Q_p)} \cM_{\infty}|, \chi)\]
for any \'etale $U\rightarrow \cM_{\infty}/B(\Q_p)$. Here one turns the right action of $B(\Q_p)$ on $\cM_\infty$ into a left action and then the subscript $B(\Q_p)$ denotes the set of all maps which are equivariant for this action. 

\begin{prop} The sheaf $\cF_{\chi}$ on $(\cM_{\infty}/B(\Q_p))_{\et}$ is a $\mathbb{F}_q$-local system of rank one.
\end{prop} 
\begin{proof}
The characters $\chi_i$ are smooth, so there exists an open subgroup $K\subset B(\Q_p)$ such that ${\chi_1}|_K= {\chi_2}|_K = 1 $ is trivial. 
Fix such $K \subset B(\Z_p)$. Then 
\[ f:\cM_{\infty}/K \rightarrow \cM_{\infty}/B(\Q_p)
\] 
is an \'etale covering and $f^*\cF_{\chi}$ is the constant sheaf $\underline{\mathbb{F}_q}$. 
\end{proof}

\begin{lem} Let $X$ be a topological space with a left action of a locally profinite group $G$ such that $G\times X \rightarrow X$ is continuous. Let $B\subseteq G$ be a closed subgroup such that $G/B$ is compact. Let $\chi: B\rightarrow \mathbb{F}^*_q$ be a smooth character and let $\Ind_B^G(\chi)$ be the smooth induction of $\chi$ from $B$ to $G$.  Then
\[\Map_{\cont, G}(X, \Ind_B^G(\chi)) = \Map_{\cont,B}(X,\chi). \]
\end{lem}
\begin{proof} This is elementary, we just give the maps and leave it to the reader to check that they are well-defined and inverse to each other.
Define 
\begin{eqnarray*}
S:\Map_{\cont, G}(X, \Ind_B^G(\chi)) &\rightarrow & \Map_{\cont,B}(X,\chi),\\
 \phi &\mapsto& S(\phi): x \mapsto \phi(x)(1),
\end{eqnarray*}
and 
\begin{eqnarray*}
T: \Map_{\cont,B}(X,\chi) &\rightarrow & \Map_{\cont, G}(X, \Ind_B^G(\chi)),\\ 
\psi &\mapsto & T(\psi)(x): g\mapsto \psi(gx). 
\end{eqnarray*}
\end{proof}
Recall the sheaf $\cF_{\pi}$ is defined as 
\[\cF_{\pi}(U) =  \Map_{\cont,\GL_2(\Q_p)}(|U\times_{\mathbb{P}^1_{\C_p}} \cM_{\infty}|, \pi)\]
for an \'etale map $U\rightarrow \mathbb{P}^1_{\C_p}$. From the above lemma we see that $\overline{\pi}_{\GH,*}\cF_{\chi}=\cF_{\pi}$. 

\begin{prop}\label{directimage} Let $\cF$ be a sheaf on $(\cM_{\infty}/B(\Q_p))_{\et}$. Then
\[H^i_{\et}(\cM_{\infty}/B(\Q_p),\cF)\cong H^i_{\et}(\mathbb{P}_{\C_p}^{1}, \overline{\pi}_{\GH,*} \cF)\]
for all $i\geq 0$.
\end{prop}
\begin{proof} We show that $R^i\overline{\pi}_{\GH,*}\cF=0$ for all $i>0$ by calculating the stalks
\[(R^i\overline{\pi}_{\GH,*}\cF)_{\xbar}\]
at any geometric point $\xbar= \spa(C(\xbar),C(\xbar)^+)$. 

By Proposition \ref{qc} we can apply \cite[Lemma 4.4.1]{CS} to get an isomorphism
\[(R^i\overline{\pi}_{\GH,*}\cF)_{\xbar} \cong H^i_{\et}((\cM_\infty/B(\Q_p))_{\xbar}, \cF),\]
where 
\[(\cM_\infty/B(\Q_p))_{\xbar} = \left(\cM_\infty/B(\Q_p)\times_{\mathbb{P}^1} \spa(C(\xbar),C(\xbar)^+)\right)^\wedge\]
is the fibre of $\cM_\infty/B(\Q_p)$ over $\xbar$, which we can identify with the perfectoid space $\overline{y}\times \underline{\mathbb{P}^1(\Q_p)}$, where $\overline{y}$ is any lift of $\overline{x}$ to $\cM_\infty$. The claim now follows from Lemma \ref{cohpt}. 
\end{proof}

\begin{lem}\label{sheaves} We have an isomorphism of sheaves on $(\mathbb{P}^1_{\C_p})_{\et}$
\[ (\overline{\pi}_{\GH,*}\cF_\chi) \otimes \cO^+_{\mathbb{P}_{\C_p}^1}/p \cong \overline{\pi}_{\GH,*}(\cF_\chi \otimes \cO^+_{\cM_\infty/B(\Q_p)}/p). \]
\end{lem}
\begin{proof} There is a natural morphism 
\[(\overline{\pi}_{\GH,*}\cF_\chi) \otimes \cO^+_{\mathbb{P}_{\C_p}^1}/p \rightarrow \overline{\pi}_{\GH,*}(\cF_\chi \otimes \cO^+_{\cM_\infty/B(\Q_p)}/p)\]
so it suffices to check that the stalks at geometric points agree. For that let $\overline{x}=\spa(K,K^+)$ be a geometric point of $\mathbb{P}^1_{\C_p}$.
Then using Lemma 4.4.1 of \cite{CS} we get
\begin{eqnarray*}((\overline{\pi}_{\GH,*}\cF_\chi) \otimes \cO^+_{\mathbb{P}_{\C_p}^1}/p)_{\overline{x}} &\cong &(\overline{\pi}_{\GH,*}\cF_\chi)_{\overline{x}} \otimes (\cO^+_{\mathbb{P}_{\C_p}^1}/p)_{\overline{x}} \\
&\cong & H^0_{\et}(\left(\cM_\infty/B(\Q_p)\right)_{\overline{x}},{\cF_\chi}|_{(\cM_\infty/B(\Q_p))_{\overline{x}}})\otimes K^+/p.
\end{eqnarray*}
As before
\[\left(\cM_\infty/B(\Q_p)\right)_{\overline{x}}\cong \overline{y}\times \underline{\mathbb{P}^1(\Q_p)},\]
where $\overline{y}$ is any lift of $\overline{x}$ to $\cM_\infty$. 
Then 
\[H_{\et}^0(\overline{y}\times \underline{\mathbb{P}^1(\Q_p)},{\cF_\chi}|_{\overline{y}\times \underline{\mathbb{P}^1(\Q_p)}} ) \cong  \varinjlim_{U} H^0_\et(U, \cF_\chi|_U)
\]
where the limit runs over all \'etale neighbourhoods $U$ of $\overline{y} \times \underline{\mathbb{P}^1(\Q_p)}$ in $\cM_\infty/B(\Q_p)$. 
We choose an isomorphism ${\mathcal{F}_\chi}|_{\overline{y}\times \underline{\mathbb{P}^1(\Q_p)}} \cong \mathbb{F}_q$. Then 
\begin{eqnarray*} 
H_{\et}^0(\overline{y}\times \underline{\mathbb{P}^1(\Q_p)},\mathbb{F}_q ) &\cong & \varinjlim_{U} H^0_\et(U, \mathbb{F}_q)\\
&\cong&  \mathrm{Map}_{\cont}(\varprojlim \pi_0(U), \mathbb{F}_q)\\
&\cong &  \mathrm{Map}_{\cont}(\mathbb{P}^1(\Q_p), \mathbb{F}_q) \\
&\cong &  \mathrm{Map}_{\mathrm{loc.cst.}}(\mathbb{P}^1(\Q_p), \mathbb{F}_q).
\end{eqnarray*}

To calculate the other side, we use Lemma 4.4.1 of \cite{CS} again to get 
\begin{eqnarray*} &&\overline{\pi}_{\GH,*}(\cF_\chi \otimes \cO^+_{\cM_\infty/B(\Q_p)}/p)_{\overline{x}} \\ &\cong& H_{\et}^0\left(\left(\cM_\infty/B(\Q_p)\right)_{\overline{x}}, (\cF_\chi \otimes \cO^+_{\cM_\infty/B(\Q_p)}/p)|_{\left(\cM_\infty/B(\Q_p)\right)_{\overline{x}}}\right).
\end{eqnarray*}
Now 
\begin{eqnarray*} & & (\cF_\chi \otimes \cO^+_{\cM_\infty/B(\Q_p)}/p)|_{(\cM_\infty/B(\Q_p))_{\overline{x}}} \\
&\cong& \cF_\chi|_{(\cM_\infty/B(\Q_p))_{\overline{x}}} \otimes \cO^+_{\cM_\infty/B(\Q_p)}/p|_{(\cM_\infty/B(\Q_p))_{\overline{x}}}\\
&\cong& \mathbb{F}_q \otimes_{\mathbb{F}_p} \cO^+_{(\cM_\infty/B(\Q_p))_{\overline{x}}}/p 
\end{eqnarray*}
and so the claim follows from
\[H^0(\overline{y}\times \underline{\mathbb{P}^1(\Q_p)},\cO^+_{(\cM_\infty/B(\Q_p))_{\overline{x}}}/p) \cong \Map_{\mathrm{loc.cst}}(\mathbb{P}^1(\Q_p),K^+/p).
\]
\end{proof}

\begin{thm}\label{mainthm} Let $\chi_i:\Q_p^* \rightarrow \mathbb{F}^*_q, i=1,2$, be two smooth characters. Then 
\[\cS^2(\Ind(\chi)) =0.\]
\end{thm}
\begin{proof}
As $\cS^2(\Ind(\chi))$ is an $\mathbb{F}_p$-vector space, it suffices to show that 
\[H_\et^2(\mathbb{P}^1_{\C_p}, \cF_\pi) \otimes \cO_{\C_p}/p \]
is almost zero.

We have a chain of almost isomorphisms 
\begin{eqnarray}
H_\et^2(\mathbb{P}^1_{\C_p}, \cF_\pi) \otimes \cO_{\C_p}/p &\cong^a& H_\et^2(\mathbb{P}^1_{\C_p}, \cF_\pi \otimes \cO^+_{\mathbb{P}_{\C_p}^1}/p) \\
& \cong & H_\et^2(\mathbb{P}^1_{\C_p}, (\overline{\pi}_{\GH,*}\cF_\chi) \otimes \cO^+_{\mathbb{P}_{\C_p}^1}/p)\\
& \cong & H_\et^2(\mathbb{P}^1_{\C_p}, \overline{\pi}_{\GH,*}(\cF_\chi \otimes \cO^+_{\cM_\infty/B(\Q_p)}/p)) \\
& \cong & H_\et^2(\cM_\infty/B(\Q_p), \cF_\chi \otimes \cO^+_{\cM_\infty/B(\Q_p)}/p),
\end{eqnarray}
where the first $\cong^a$ is \cite[Theorem 3.2]{scholzeLT}, the second and last are implied by Proposition \ref{directimage}, and the third is Lemma \ref{sheaves}.

Consider the morphism of sites
\[\overline{\pi}_{\GH}:(\cM_\infty/B(\Q_p))_{\et} \rightarrow (\mathbb{P}^1_{\C_p})_{an}.\]
We claim that 
\[ H^i_{\et}(\cM_{\infty}/B(\Q_p), \cF_\chi \otimes \cO^+/p) \cong^a H^i_{an}(\mathbb{P}^1_{\C_p}, \overline{\pi}_{\GH,*}(\cF_\chi \otimes \cO^+/p)).\]
Just as in \cite[Lemma 4.4.1]{CS} one proves that for a point $x=\spa(K,K^+)$ of $\mathbb{P}^1_{\C_p}$ one has
\[(R^i\overline{\pi}_{\GH,*}\cF_\chi \otimes \cO^+/p)_x \cong H^i_{\et}(\overline{\pi}^{-1}_{\GH}(x),\cF_\chi \otimes \cO^+/p).\]
By Lemma \ref{affinoidfibres}, $\overline{\pi}^{-1}_{\GH}(x) \subset \cM_\infty/B(\Q_p)$ is affinoid perfectoid and so by \cite[Lemma 4.12]{scholzepHT}
\[H^i_{\et}(\overline{\pi}^{-1}_{\GH}(x), \cF_\chi \otimes \cO^+/p)^a\] 
is almost zero for all $i>0$, which shows the above claim. 

But for any sheaf $\mathcal{F}$, $H^i_{an}(\mathbb{P}^1,\mathcal{F})= 0$ for $i>1$ (see \cite[Proposition 2.5.8]{dJvdP}). This finishes the proof.
\end{proof}

Note that in the above theorem we did not assume that the characters are distinct. As the functor $\pi \mapsto \cF_\pi$ is exact and as $\cS^3=0$ (\cite[Theorem 3.2]{scholzeLT}) we also get the following corollary.
\begin{corollary} Let $St$ denote the Steinberg representation, and let $\mu:\Q_p^* \rightarrow \mathbb{F}^*_q$ be a smooth character. Then 
\[\cS^2(St \otimes \mu) = 0. \]
\end{corollary}

\begin{rem} Finally we discuss $\cS^1$. Let $\chi_i: \Q_p^*\rightarrow \mathbb{F}_q^*, i=1,2$ be two smooth characters, such that 
\[\chi_1 \neq \chi_2, \chi_2\omega, \chi_2\omega^2.\]
It is expected that for such pairs $(\chi_1,\chi_2)$
 \[\cS^1(\Ind(\chi_1,\chi_2))\neq 0
\] 
and is even infinite-dimensional. We give some details. By \cite[Theorem 2.1]{breuilicm} there is a unique non-split extension 
\[0\rightarrow \Ind(\chi_1,\chi_2) \rightarrow \Pi \rightarrow \Ind(\chi_2\omega, \chi_1 \omega^{-1})\rightarrow 0.\]
Let $\overline{r}:G_{\Q_p}\rightarrow \GL_2(\mathbb{F}_q)$ be the Galois representation associated with $\Pi$ under the mod~$p$ local Langlands correspondence, i.e., the unique non-split extension in $\mathrm{Ext}_{\Gal(\Qbar_p,\Q_p)}^1(\chi_2\omega, \chi_1)$ (\cite[Definition 2.2]{breuilicm}). Note this is an indecomposable representation and $\mathrm{End}(\overline{r})=\{1\}$.

One can globalize $\overline{r}$ to an absolutely irreducible Galois representation 
\[\overline{\rho}:G_F\rightarrow \GL_2(\overline{\mathbb{F}}_p),\]
where $F$ is a totally real field of even degree $[F:\Q]$ such that the prime $p$ is totally split in $F$ and such that $\overline{\rho}$ is automorphic for the quaternion algebra $D_0/F$ which is split at all finite places and definite at all archimedean places. In particular, 
\[\overline{\rho}|_{G_\mathfrak{p}}=\overline{r}\]
for any prime $\mathfrak{p}$ above $p$.
This follows from \cite[Corollary A.3]{geekisin}, an additional base change and the fact that a RAESDC automorphic representation of $\GL_2(F)$ is in the image of the global Jacquet--Langlands transfer from $D_0$.   

Let $G$ be the algebraic group defined by the units in $D_0$. The Galois representation~$\overline{\rho}$ gives a maximal ideal $\mathfrak{m}$ of some Hecke algebra and we fix a tame level $U^p \subset G(\A_{F,f}^p)$ such that 
\[\pi_{U^p}[\mathfrak{m}]\neq 0.\]
Here $\pi_{U^p}$ denotes the completed cohomology group as in \cite[Definition 6.1]{scholzeLT}. Scholze shows that in this situation the space
\[ \mathcal{S}^1(\pi_{U^p}[\mathfrak{m}])
\]
is infinite-dimensional. It is expected that $\pi_{U^p}[\mathfrak{m}]$ is of finite length and only has principal series representations as subquotients. In fact $\pi_{U^p}[\mathfrak{m}]$ should only have copies of $\Pi$ in its cosocle \cite[Remark 7.8]{sixauthors2}. Granting this, we get that for at least one constituent $\pi$ of $\Pi$, the space $\mathcal{S}^1(\pi)$ is non-zero and even infinite-dimensional. So the mod $p$ Jacquet--Langlands correspondence behaves differently from the classical Jacquet--Langlands correspondence.    
\end{rem}

\bibliography{LT}
\bibliographystyle{plain}

\end{document}